\documentclass[10pt]{amsart}

\usepackage{amsfonts,amssymb,amsmath,amscd,amstext}
\usepackage[colorlinks=true,linkcolor=blue,citecolor=blue]{hyperref}
\usepackage{microtype}
\usepackage{graphicx}
\usepackage[dvipsnames]{xcolor}

\renewcommand{\leq}{\leqslant}

\newcommand{\rr}{{\mathbb{R}}}

\newcommand{\hh}{{\mathbb{H}}}
\newcommand{\GG}{{\mathbb{G}}}
\newcommand{\nn}{{\mathbb{N}}}

\newcommand{\hx}{{\hat x}}
\newcommand{\hxi}{{\hat \xi}}
\newcommand{\hy}{{\hat y}}
\newcommand{\he}{{\hat \eta}}
\newcommand{\hv}{{\hat v}}
\newcommand{\hz}{{\hat z}}

\newcommand{\eps}{\varepsilon}
\newcommand{\ga}{\gamma}

\newcommand{\escpr}[1]{\langle#1\rangle}
\newcommand*{\medcup}{\mathbin{\scalebox{0.8}{\ensuremath{\cup}}}}

\DeclareMathOperator{\divv}{div}

\newtheorem{theorem}{Theorem}[section]
\newtheorem{proposition}[theorem]{Proposition}
\newtheorem{lemma}[theorem]{Lemma}
\newtheorem{corollary}[theorem]{Corollary}

\theoremstyle{definition}

\newtheorem{remark}[theorem]{Remark}
\newtheorem{example}[theorem]{Example}

\newtheorem{definition}[theorem]{Definition} 
\numberwithin{equation}{section}
\theoremstyle{remark}

\begin{document}

\title[Schauder estimates up to the boundary on h-type groups]{Schauder estimates up to the boundary on h-type groups: an approach via the double layer potential}

\author[G.~Citti]{Giovanna Citti}
\address{Dipartimento di Matematica, Piazza di Porta S. Donato 5, 401
26 Bologna, Italy}
\email{giovanna.citti@unibo.it}

\author[G.~Giovannardi]{Gianmarco Giovannardi}
\address{Dipartimento di Matematica Informatica "U. Dini", Universita  degli Studi di Firenze, Viale Morgani 67/A, 50134, Firenze, Italy}
\email{gianmarco.giovannardi@unifi.it}

\author[Y.~Sire]{Yannick Sire}
\address{Department of Mathematics, Johns Hopkins University, Baltimore, MD 21218, USA}
\email{ysire1@jhu.edu}

\date{\today}

\thanks{The first author have been supported by Horizon 2020 Project ref. 777822: GHAIA. The second author have been supported by MEC-Feder grant PID2020-118180GB-I00, Horizon 2020 Project ref. 777822: GHAIA and INdAM-GNAMPA Project 2022 CUP-E55F22000270001. The third author is partially funded by NSF grant DMS 2154219 and the Simons foundation. }

\maketitle

\thispagestyle{empty}

\begin{abstract}
We establish the Schauder estimates at the boundary away from the characteristic points for the Dirichlet problem
by means of the double layer potential in { a Heisenberg-type group $\GG$.} 
Despite its singularity we manage to invert the double layer potential restricted to the boundary thanks to a reflection technique for an approximate operator in {$\GG$. }
This is the first instance where a reflection-type argument appears to be useful in the sub-Riemannian setting.
\end{abstract}

\section{Introduction}
Schauder estimates at the boundary are important tools in regularity theory and applications to PDEs. On an Euclidean (smooth) domain, they are obtained classically by combining a  flattening of the boundary together with  a reflection argument that allows to use interior estimates.

This classical reflection technique does not work when the ambient space, instead of $\rr^n$, is the simplest sub-Riemannian manifold: the Heisenberg $\mathbb{H}^n$, $n\ge1$. Given $f,g \in C^{\infty} $ Kohn and Nirenberg in \cite{KN65} proved that the solution is smooth up to the boundary at the non-characteristic points of the boundary $\partial \Omega$, where the projection of the Euclidean normal to $\partial \Omega$ onto the horizontal  distribution is different from zero.
Then Jerison in \cite{Jerison1} using  the method of the \textit{single layer potential} was able to invert the operator restricted to the boundary to construct a Poisson kernel that directly yields  to the Schauder estimates around a non-characteristic point.
Later in \cite{Jerison2} under suitable conditions on the characteristic point  Jerison obtained the Schauder estimates around a strongly isolated characteristic boundary  point. On the other hand, always in \cite[Section 3]{Jerison2} Jerison showed the celebrated example of the paraboloid where the Schauder estimate fails, since  the solution
of the Dirichlet problem with real analytic datum $g$ and $f=0$ may be not better than h\"older continuous near a
characteristic boundary point.

Recently Baldi, Citti and Cupini \cite{BCC19}, assuming a geometric hypothesis on the boundary, obtained  Schauder estimates at the boundary in neighborhood of non-characteristic points for the problem $\Delta_{\mathbb{G}} u= f$ in  $\Omega$ with Dirichlet boundary condition 
$u=g $ on $\partial \Omega$.
Here $ \Delta_{\mathbb{G}}= \sum_{i=1}^k X_i^2$ is   the  sub-Laplacian in a generic Carnot group $\mathbb{G}$ with distribution $V^1$ generated by the vector fields  $X_1,\ldots,X_k$ and  $\Omega $ is a bounded open set of $\mathbb{G}$. Here $f$ and $g$ belong to suitable classes of  h\"{o}lderian functions. The additional hypothesis that Baldi, Citti and Cupini \cite{BCC19} assume is  that the induced distribution on the boundary, generated by the vector fields tangent to the boundary that belongs to the distribution $V^1$, verifies the H\"{o}rmander condition.
However even in the simplest case of the Heisenberg group $\mathbb{H}^1$ this hypothesis is not verified. Another relevant paper concerning the $C^{1,\alpha}$ Schauder estimates in Carnot groups near a non-characteristic portion of the boundary of regularity $C^{1,\alpha}$ is \cite{MR3948989}. The full Schauder estimates have been subsequently obtained in the recent paper \cite{2022arXiv221012950B} where the authors obtain  estimates in $\Gamma^{k,\alpha}$ for $k\ge2$ near a $C^{k,\alpha}$ non-characteristic portion of the boundary in a Carnot group by means of a compactness method going back to seminal works of  Caffarelli (see e.g. \cite{MR1005611}).\footnote{A first version of the current paper was only considering the case of the first Heisenberg group. We decided in this new version to deal with all possible H-type groups; the paper \cite{2022arXiv221012950B} appeared while this work was in preparation.}

The aim of this work is to establish the Schauder estimates at the boundary away from the characteristic points for the Dirichlet problem
\begin{equation}
\label{eq:PP}
\begin{cases}
\Delta_{\GG} u= f & \quad \text{in} \quad \Omega\\
u= g & \quad \text{on} \quad \partial \Omega
\end{cases}
\end{equation}
by means of the \textit{double layer potential} in a group of Heisenberg type $\GG$. Here $\Omega \subset \GG$ is a bounded open set and $f\in C^{\alpha}(\bar{\Omega})$, $g \in C^{2,\alpha}(\partial \Omega)$. Since the function $f$ in \eqref{eq:PP} concerns only interior estimates, it is not restrictive to study the associated homogeneous problem $\Delta_{\GG} u=0$ in $\Omega$ and $ u=g$  on $\partial \Omega$. By the Green's representation formula the solution of the previous problem is given by
\[
u(x)=\int_{\partial \Omega} g(y) \escpr{ \nabla_{\GG} G(x,y), \nu_h(y)} d\sigma(y),
\]
where $\nabla_{\GG}$ is the horizontal gradient and  $G(x,y)$ is the Green's function such that  $G=0$ on $\partial \Omega $. Then we consider an harmonic approximation of  $u(x)$ given by $\mathcal{D}(g)(x)$ where  instead of  $\nabla_{\GG} G(x,y)$ we consider $\nabla_{\GG} \Gamma(x,y)$, where  $\Gamma$ is the fundamental solution for $\Delta_{\GG}$. We will call $\mathcal{D}(g)$ the \textit{double layer potential}. Clearly the harmonic function $\mathcal{D}(g)$ does not assume the boundary datum $g$, but  when $x \rightarrow \xi \in \partial \Omega$ and $\xi$ is a non-characteristic point, we get that the limit of  the double layer potential coincides  with $T(g)(\xi):=\frac{1}{2} g(\xi)+ K(g)(\xi)$, where $K$ is a singular operator from  $C^{2,\alpha}$ into $C^{2,\alpha}$. Jerison in \cite{Jerison1} pointed already the singularity of $K$ and chose to use instead the  single layer potential to derive Schauder estimates up to the boundary. On one hand in the Euclidean case the analogous of  $K$ is a compact operator if the boundary is smooth, thus $T$ is invertible, by the Fredholm alternative. Then,  we have that $\mathcal{D}(T^{-1}(g))$ is harmonic and assume the boundary datum. Hence the  Schauder estimates follow automatically by the  h\"{o}lder estimates  for $T$. On the other hand when the boundary   $\partial \Omega$ is  Lipschitz the operator $K$ is  singular  as well as in the sub-Riemannian case. However, Verchota in \cite{MR769382} inverted  the operator $T$ between $L^2$ classes on the Lipschitz boundary in an Euclidean domain.  He  showed that  the $L^2$ norm of  $K$ is small compared to $\tfrac{1}{2}$ extending  the previous result by Fabes, Jodeit, and Rivi\'ere \cite{MR501367} for $C^1$ domains in $\rr^n$. The technique developed in \cite{MR769382} was influenced by the previous seminal papers by Calder\'on \cite{MR466568} and  Coifman, McIntosh, and Meyer \cite{MR672839} which  established the $L^p$-boundedness of the Riesz transform and the double layer potential in this Lipschitz Euclidean setting. All these results are very well explained and (some of them) extended in the book \cite{KC94} by Kenig.  It is important to keep in mind that in this Euclidean setting, the double layer potential is singular since the domain is rough. Several extensions of the boundary layer technique have been developed over the last years in a variety of settings (see e.g. \cite{kim} and references therein for recent developments). 

In the present situation, as mentioned above, the situation is more dramatic since the double layer potential is singular {\sl even} on smooth domains. Fortunately in the present work we do not need to develop an  $L^2$ theory for singular integrals in the sub-Riemannian setting as Orponen and Villa did in \cite{2020arXiv200608293O}; instead we need to prove that $K$ has small $C^{2,\alpha}$ norm with respect to $\frac{1}{2}$ in order to use the continuity method developed by \cite{MR769382} (see also \cite[page 56]{KC94}). Crucial to our strategy is a reflection argument, special to the sub-Riemannian setting, which we believe will be proved to be useful for other problems when the lack of commutativity is critical.

We now describe our strategy in the case of the  upper half space $\{ x_1>0\}$ of the first Heisenberg group $\hh^1$ with coordinates $(x_1,x_2,x_3)$ in $\hh^1$. The boundary of  $\{ x_1>0\}$ is given by the intrinsic plane $\Pi=\{x_1=0\}$, so called in literature since it does not contain characteristic points. In this case the singular operator $K$ is an operator on $\Pi$ with convolution kernel $k$ given by \eqref{eq:kk}.
In order to  prove that $K$ has small $C^{2,\alpha}$ norm with respect to $\frac{1}{2}$ we use a surprising reflection technique in the Heisenberg group. Indeed the main obstacle in applying the reflection technique in the Heisenberg group  is the fact that if $u$ satisfies an equation, its reflection $u(-x_1, x_2, x_3)$ does not satisfy the 
same equation, because of the non commutation properties of $X_1$ and $X_2$. As a consequence, while the increments appearing in the kernel along the horizontal directions are the standard ones, $x_1- y_1$ and $ x_2 - y_2 $, in the third increment mixed variables show up
$$x_3 - y_3 - \frac{1}{2}(y_1 x_{2}-y_{2}x_1).$$
However  this increment, restricted to the boundary $x_1=y_1=0$, reduces to the standard one, 
and all variables decouple. This allows to apply a reflection technique, if we are interested in the limit, 
when the operators tend to the boundary. 
Hence we modify the operators on the whole space, removing the mixed term, which is not present in the 
limit, and is the cause of preventing reflection. Hence the symmetry of these new operators on the whole space yields that in the limit the $C^{2,\alpha}$ norm of $\tfrac{1}{2} I+ K$ and  
$-\tfrac{1}{2} I+ K$ coincide. Finally the continuity method allows to prove the invertibility of $\tfrac{1}{2} I+ K$, thus the solution of the Dirichlet problem on the half space is given by harmonic function $\mathcal{D}((\tfrac{1}{2} I+ K)^{-1}g)$. Once we obtain the invertibility of $\tfrac{1}{2} I+ K$ on the flat space  the general case for curved domains follows directly. Indeed flattening the boundary around a non-charatheristic point involves a compact operator $K_{\hat R}$ that does not affect the invertibility  $\tfrac{1}{2} I+ K+K_{\hat R}$. Hence the  Schauder estimates around a non-charatheristic point are a direct consequence of the H\"older estimates for the inverse of $\tfrac{1}{2} I+ K+K_{\hat R}$.

Even if the result in the present paper, i.e. the Schauder estimates at the boundary on $\hh^1$, is known since the works of Jerison and by now thanks to \cite{2022arXiv221012950B}, in the full generality of any Carnot groups, we want to emphasize  that it was not known that one could achieve such results via the double layer potential. We believe that our method, being independent of all the previous ones, will be proved to be useful for other problems and provide a new point of view on singular integrals in groups. 

In order to simplify the exposition, we concentrate in the first five sections  only on the case of  $\hh^1$, for which the computations are easier. In the last section{ \ref{sc:GHT} we show that the same technique provides the Schauder estimates  in the more general setting of $H$-type groups. We emphasize that for the  $C^{2,\alpha}$ estimates of the singular operator $K$ (Theorem \ref{th:Kcont} ) in the case of $\hh^1$ we provide a standard  proof based on \cite[Lemma 4.4]{GT} in order to make the exposition self-contained, whereas in the case of the $H$-type group $\GG$ (Theorem \ref{th:KcontH}) we used a deep result by Nagel and Stein  \cite{NagelStein79}}.

{To be more precise on the boundary $\partial \Omega$ the h\"older classes are defined by  the non-isotropic Folland-Stein h\"older classes $\Gamma^{2,\alpha}(\partial \Omega)$ introduced by Jerison \cite[Section 4]{Jerison1}, see also \cite{NagelStein79,MR657581}. Roughly speaking a function belongs to $\Gamma^{2,\alpha}(\partial \Omega)$ if at each scale $\delta$ we approximate the function by a second order polynomial in local coordinates on $\partial \Omega$ with an error that goes as $\delta^{2+\alpha}$, this idea in the Euclidean setting goes back to Campanato \cite{MR167862}.  We show that on the plane $\Pi=\{x_1=0\}$ the class $\Gamma^{2,\alpha}(\Pi)$ coincides with the classical sub-Riemannian h\"older space $C^{2,\alpha}(\Pi)$, that means that the second order  horizontal tangential derivatives and the first order vertical derivatives are h\"older continuous with respect to the induced distance on $\Pi$. The control on the first order vertical derivatives is crucial since on $\Pi$ we do not verify the H\"ormander rank condition such as in \cite{BCC19}, thus the vertical derivatives are not a priori commutators of vector fields that belong to the distribution.  To our knowledge the equivalence of these two classes of functions was known only when instead of the intrinsic plane $\Pi$ we consider the whole group $\GG$, see \cite[Theorem 5.3]{MR657581}. 
}

{Finally we point out that the Schauder estimates for general $H$-type groups, obtained in Section \ref{sc:GHT}, are not consequences of the results in  \cite{BCC19} (but of course follow from the results in \cite{2022arXiv221012950B}); indeed there are several examples of $H$-type groups, different from $\hh^1$, that does not satisfy the H\"ormander rank condition on $\Pi$, see for instance  Example \ref{eq:nothormander}. As is usual in layer potential methods, one needs to have a rather precise formula for the fundamental solution of the operator to be able to conclude. This is the main issue in our technique to deal with general Carnot groups.  
}

The paper is organized as follows. In Section \ref{sc:pre} the Heisenberg group, the fundamental solution for the sub-Laplacian and the H\"older classes  are introduced.
Section \ref{sc:dlp} deals with the double layer potential and its jump formulas.  In Section \ref{sc:invert} we provide the  invertibility of the double layer potential on the intrinsic plane by the reflection technique.  In Section \ref{sc:Schest} we show the local Schauder estimates around a non-characteristic point and the global Schauder estimates for  bounded domains without characteristic points. Examples of such domains are constructed in \cite{MR4219401}.
{ Finally in Section \ref{sc:GHT} we show that all the previous results hold in the more general setting of $H$-type groups.}
\section{Preliminaries}
\label{sc:pre}
The first Heisenberg group $\hh^1$ is an analytic, simply
connected  $3$-dimensional Lie group  such that its Lie algebra $\mathfrak{g}$ admits a stratification
\[
 \mathfrak{g}=V^1\oplus V^2, \quad [V^1,V^1]=V^2 \quad \text{and}  \quad   [V^1,V^{2}]=\{0\}.                                                                                                                                                                                                                                   \]
The stratification induces a natural notion of degree of a vector field
\[
\deg(X)=j  \quad \text{whenever} \quad X \in V_j,
\]
for $j=1,2$.
By \cite[Theorem 2.2.18]{BLU} the first Heisenberg group $\hh^1$ can be identified with the triple $(\mathbb{R}^3, \circ, \delta_{\lambda})$, where $\circ$ is the polynomial group law given by 
\[
y\circ x =\Bigg(y_1+ x_1, y_{2}+x_{2}, y_3+x_3 + \frac{1}{2} (y_1 x_{2}-y_{2}x_1)\Bigg),
\]
for any pair of points $x=(x_1,x_2,x_3)$, $y=(y_1,y_{2},y_3)$ in $\rr^{3}$ and  $\{\delta_{\lambda}\}_{\lambda>0}$ is a family of automorphisms of $(\rr^3, \circ)$ such that 
\[
\delta_{\lambda} (x_1, x_2, x_3)= ( \lambda x_1,  \lambda x_2, \lambda^{2} x_3 ).
\]
The homogenous dimension $Q$ is given by 
\[
Q:= \dim(V^1)+ 2 \dim(V^2)= 4 .
\]
We  call horizontal distribution the subspace $V^1$ and we choose the basis of left invariant vector fields 
\[
X_1=\partial_{x_1}- \dfrac{x_2}{2} \partial_{x_3}, \qquad X_2=\partial_{x_2}+ \dfrac{x_1}{2} \partial_{x_3}.
\] 
We denote by $\nabla_{\hh}$ the horizontal gradient 
$
\nabla_{\hh}= (X_1,X_2)
$
and by $\nabla=(\partial_1,\partial_2,\partial_3)$ the standard Euclidean gradient.  The sub-Laplacian operator is given by
\[
\Delta_{\hh}=  X_1^2+ X_2^2= \divv_{\hh}(\nabla_{\hh} ),
\]
where $\divv_{\hh}(\phi)= X_1(\phi_1)+X_2(\phi_2)$ for $\phi=\phi_1 X_1+\phi_2 X_2 \in V^{1}$. Is is well known (see \cite[Chapter 5]{BLU}) that the sub-Laplacian admits a unique fundamental solution $\hat{\Gamma} \in C^{\infty}(\rr^3 \smallsetminus \{0\})$, $\hat{\Gamma} \in L_{\text{loc}}^1(\rr^3)$, $\hat{\Gamma}(x) \to 0$ when  $x$  tends to infinity and such that 
\[
\int_{\rr^3} \hat{\Gamma}(x) \,  \Delta_{\hh}   \varphi(x) \,  dx = -\varphi(0) \quad \forall \varphi \in C^{\infty}(\rr^3).
\]
\begin{definition}
We call Gauge norm on $\hh^1$ a homogeneous symmetric norm $d$ smooth out of the origin and satisfying
\[
\Delta_{\hh} (d(x)^{2-Q})=0 \quad \forall x \ne 0 .
\]
\end{definition}
Following \cite{MR315267,MR1219650} a Gauge norm in $\hh^1$ is given by 
\[
|x|_{\hh}=((x_1^2+x_2^2)^2+ 16 x_3^2)^{\tfrac{1}{4}}.
\]
Therefore we have 
\[
\hat{\Gamma}(x)=(2 \pi)^{-1} |x|_{\hh}^{2-Q}= \tfrac{1}{2 \pi \Big( \big( x_1^2 + x_2^2\big)^2 + 
16 x_3^2 \Big)^{1/2}}.
\]
Finally we define the fundamental solution  $\Gamma(x, y) = \hat \Gamma(y^{-1} \circ x)$ that is given by 
\begin{equation}
\label{eq:fundsol}
\Gamma(x, y)=
\tfrac{1}{ 2 \pi \Big( \big( (x_1-y_1)^2 + (x_2-y_2)^2\big)^2 + 
16 \big(x_3 - y_3 - \frac{1}{2}  (y_1 x_{2}-y_{2}x_1) \big)^2 \Big)^{1/2}}
\end{equation}
and the Gauge distance $d(x,y)=|y^{-1} \circ x|_{\hh}$ for all $x,y \in \mathbb{H}^1$ .
The Gauge ball  center at  $x \in  \hh^1 $ of radius $r>0$ is given $B_r(x):=\{y \in \mathbb{H}^1 \ : \ d(x,y)< r \}$.
\begin{definition}
 Let $0<\alpha<1$, $\Omega\subset\mathbb{H}^1$ be an open set and $f: \Omega \rightarrow \rr$ be function on $\Omega$. We say that $f \in C^{\alpha}(\Omega)$ if there exists a positive constant $M$ such that for every 
 $x,y$ in $\Omega$ 
 \[
  |f(x)-f(y)|<M\ d^{\alpha}(x,y). 
 \]
 We set 
 \[
  \lVert f \rVert_{C^{\alpha}({\Omega})}=\sup_{x\ne y}\dfrac{|f(x)-f(y)|}{d^{\alpha}(x,y)}+ \sup_{x \in \Omega} |f(x)|.
 \]
 Iterating this definition, when $k>1$ we say that $f \in C^{k,\alpha}(\Omega)$ if $X_i u \in C^{k-1,\alpha}(\Omega)$ for all $i=1,2$.
\end{definition}

\subsection{Smooth domains, characteristic points and holder classes on the boundary}
\begin{definition}
\label{def:smoothboundary}
The set $\Omega$ is called a domain of class $C^{\infty}$ if for each $\xi \in \partial \Omega$ then there exists a neighborhood $U_{\xi}$ and a function $\psi_{\xi} \in C^{\infty} (U_{\xi})$ such that
\begin{align*}
U_{\xi} \cap \Omega &=\{ x \in U_{\xi} \ : \  \psi(x)<0  \}\\
U_{\xi} \cap \partial \Omega&=\{ x \in U_{\xi} \ : \  \psi(x)=0  \}.
\end{align*}
We say that $\xi$ in $\partial \Omega$ is a \textit{characteristic point} if $\nabla_{\mathbb{H}} \psi (\xi)=0$.
\end{definition}

Let $\xi, \eta$ in $\partial \Omega$, we define the induced distance $\hat{d}$ on $\partial \Omega$ by
\[
\hat{d}(\xi,\eta):=d(\xi,\eta),
\]
where $d$ is the Gauge distance in $\mathbb{H}^1$ and for $r>0$ we call $\hat B_r( \xi)$ the induced ball given by 
\[
\hat B_r( \xi)=B_r(\xi) \cap \partial \Omega,
\]
where $B_r(\xi)$ is a Gauge ball in $\hh^1$ centered at $\xi$.
\begin{definition}
Let $0<\alpha<1$. We say that a continuous function $f$ belongs to $C^{\alpha}(\partial \Omega)$ if  there exists a constant $C$ such that 
\[
\dfrac{|f(\xi)-f(\eta)|}{ \hat{d}(\xi, \eta)^{\alpha}} < C,
\]
for each $\xi, \eta$ in $\partial \Omega$. Then the holder semi-norm $[f]_{\alpha}$ is defined  by 
\[
[f]_{\alpha}= \sup_{\substack{\xi ,\eta  \in \partial \Omega \\ \xi \neq \eta} } \dfrac{|f(\xi)-f(\eta)|}{ \hat{d}(\xi, \eta)^{\alpha}}
\]
and  the holder norm is defined by
\[
\|f \| _{\alpha}= [f]_{\alpha}+ \sup_{\xi \in \partial \Omega} |f(\xi)|
\]
\end{definition}

\begin{definition} \label{def:CKalpha} Let $0<\alpha<1$ and $k \in \nn \cup \{0\}$. We say that a bounded function $f$ belongs to $\Gamma^{k,\alpha}(\partial \Omega)$ if for each $\xi \in \partial \Omega$ and $\delta>0$ there exist a polynomial $P_{\xi} (\eta)$ of degree $k$ in local coordinates on $\partial \Omega$  and a uniform constant $C$ such that 
\[
|f(\eta)-P_{\xi}(\eta)| \le C \delta^{k+\alpha}, \qquad \hat{d}(\eta,\xi) <\delta
\]
Then the H\"older semi-norm $[f]_{k,\alpha}$ is the least possible $C$ above $+$ the supremum of the coefficient of $P_{\xi}$ and  the holder norm is defined by
\[
\|f \| _{\Gamma^{k,\alpha}}=[f]_{k,\alpha}+ \sup_{\xi \in \partial \Omega} |f(\xi)|,
\]
\end{definition}
These classes are the non-isotropic Folland-Stein h\"older classes (see \cite{MR657581} or  Section $ 9$ of \cite{NagelStein79}) introduced by Jerison \cite[Section 4]{Jerison1}.

\subsection{Polynomial in local coordinates far from the characteristic points}
Let $\xi$ in $\partial \Omega$ be a non-characteristic point and $\psi$ be the defining function of the boundary, see Definition \ref{def:smoothboundary}. Then the horizontal normal to $\partial \Omega$ is defined by 
\[
\nu_h= \dfrac{\nabla_{\hh} \psi }{|\nabla_{\hh} \psi|}.
\]
Then there exists an orthonormal frame $Z, S$
 tangent to $\partial \Omega $  such that  $\deg(Z)=1$ and $\deg(S)=2$.
 Then we consider the exponential map 
 \begin{equation}
 \label{eq:expexp}
 \begin{aligned}
 (x_1,x_2,x_3)=\exp(v_1 \nu_h ) \, \circ \, \exp \left( v_2 Z\right) \, \circ \, \exp \left( v_3 S \right) (\xi)
 \end{aligned}
 \end{equation}
On the neighborhood $U \subset \hh^1$ of $\xi$  we consider the local coordinates $v=(v_1,v_2,v_3)$ given by the inverse map $\Xi_{\xi}$ of the exponential map defined in \eqref{eq:expexp}. In the literature, these coordinates are commonly called  exponential or canonical coordinates of the second kind, see \cite{Bellaiche}. In these new coordinates $\Xi_{\xi}(U\cap \partial \Omega) \subset \{v_1=0\}$.  Then we set $\hat{v}=(v_2,v_3)$ and  $J=(j_2,j_3)$
$$\deg(J)=j_2+2 \cdot j_3$$
and
\[
\hat{v}^{J}=v_2^{j_2} v_3^{j_3}.
\]
A polynomial of order $k$ in local coordinates on the boundary is given by 
\[
P(\hat{v})=\sum_{\deg(J) \le k} a_J \hat{v}^J,
\]
where $a_J$ are constants. Hence assuming that $\eta= \Xi_{\xi}^{-1} (\hat{v})$ we set $P_{\xi}(\eta):=P(\hat{v})$ in Definition \ref{def:CKalpha}.

\section{Double layer potential}
\label{sc:dlp}
Let $\Omega$ be a bounded smooth domain of  $\hh^1$. We consider the following Dirichlet problem
\begin{equation}
\begin{cases}
\Delta_{\mathbb{H}} u=0 & \text{in} \quad \Omega \\
u=g & \text{on} \quad \partial \Omega.
\end{cases}
\label{DP}
\end{equation}
Let $h_x(y)$ be an harmonic function in $\Omega$ such that 
\begin{equation}\label{harmonic}
h_x (\cdot) |_{\partial \Omega}= \Gamma(\cdot,x)|_{\partial \Omega},
\end{equation} then the Green function $G(x,y)$ is given by 
\[
G(x,y)=\Gamma(x,y)-h_x(y).
\]
Under the assumptions ($\partial \Omega$ smooth, negligible surface measure of the singular set and  uniform exterior ball property) by \cite{UL97} a solution of \eqref{DP} is 
\begin{equation}
\label{eq:u}
u(x)=\int_{\partial \Omega} g(y) \escpr{\nabla_{\hh}^{y}G(x,y), \nu(y)} d\sigma(y),
\end{equation}
where $\nu(y)$ is the unit normal at the point $y \in \partial \Omega$.
Following \cite{KC94} we consider the approximation of  \eqref{eq:u} given by
\begin{equation}
\label{eq:w}
\mathcal{D}(g)(x)= \int_{\partial \Omega} g(y) \escpr{\nabla_{\hh}^y \Gamma (x,y), \nu(y)} d\sigma(y),
\end{equation}
proposed  by C. Neumann in the classical setting. Clearly we have that  $\Delta_{\hh} \mathcal{D}(g)(x)=0$ for each $x \in \Omega$. 

\subsection{The jump formulas across an intrinsic plane}
Let $\Omega=\{x_1>0\} \subset \hh^1$ and $\partial  \Omega =\{x_1=0\}= \Pi$. Then the induced distance $\hat d$ is given by 
\begin{equation}
\hat d(\hat x,\hat y)=((x_2-y_2)^4+ 16 (x_3-y_3)^2)^{\tfrac{1}{4}}
\end{equation}
for each $\hat x =(x_2,x_3)$ and $\hat y =(y_2,y_3)$ in $\Pi$ and  the induced ball is given by 
\[
\hat B_r(\hat x)=\Big\{ (y_2,y_3) \in \Pi \ : \ \hat d(\hat x,\hat y)<r\Big\}.
\]
\begin{proposition}
\label{prop:K1Kplane}
Let $\Omega=\{x_1>0\} \subset \hh^1$ and $\partial  \Omega =\{x_1=0\}= \Pi$.  Then the double layer potential
$ \mathcal{D}(g)(x)$ is given by 
\begin{equation}
\label{eq:DgPi}
\mathcal{D}(g)(x)= K_1(g)(x) + K(g)(x)
\end{equation}
for $x\in \Omega$, where $K_1$ and $K$ are operators with kernels 
respectively $k_1$ and $k$ defined as
\begin{equation}\label{k1}k_1(x,\hat y)=
 \dfrac{1}{\pi} \frac{\big( x_1^2 + (x_2-y_2)^2\big) x_1 }
{\Big( \big( x_1^2 + (x_2-y_2)^2\big)^2 
+ 16 \big(x_3 - y_3 + \frac{1}{2}y_{2}x_1 \big)^2 \Big)^{3/2}}
\end{equation}
\begin{equation}\label{k}k(x,\hat y)=
-\dfrac{4}{\pi}\frac{  (x_2-y_2)(x_3 - y_3 - \frac{1}{2}y_{2}x_1 )}
{\Big( \big( x_1^2 + (x_2-y_2)^2\big)^2 
+ 16 \big(x_3 - y_3 + \frac{1}{2} y_{2}x_1 \big)^2 \Big)^{3/2}},
\end{equation}
where $\hat y=(y_2,y_3)$ and $(0,\hat y) \in \Pi$.
\end{proposition}
\begin{proof}
By left invariance an explicit computation shows that the derivative \eqref{eq:fundsol} with respect to $X^x_1$ is given by
\begin{align*}
&X^x_{1}(\Gamma(x, y)) = (X_{v_1}\Gamma)(y^{-1} x) =\\
&=
-\frac{ \big( (x_1-y_1)^2 + (x_2-y_2)^2\big) (x_1-y_1) - 
4 (x_2-y_2)\big(x_3 - y_3 - \frac{1}{2} (y_1 x_{2}-y_{2}x_1)\big)}
{ \pi \Big( \big( (x_1-y_1)^2 + (x_2-y_2)^2\big)^2 + 
16 \big(x_3 - y_3 - \frac{1}{2} (y_1 x_{2}-y_{2}x_1)\big)^2 \Big)^{3/2}}
\end{align*}
Since $\Gamma$ is symmetric we also have 
\begin{align*}
&X^y_{1} \Gamma(x,y) =\escpr{\nabla_{\hh}^y \Gamma(x,y), X_1^y}\\
&=
-\frac{ \big( (x_1-y_1)^2 + (x_2-y_2)^2\big) (y_1-x_1) - 
4 (y_2-x_2)\big(y_3 - x_3 - \frac{1}{2} (x_1 y_{2}-x_{2}y_1)\big)}
{ \pi \Big( \big( (x_1-y_1)^2 + (x_2-y_2)^2\big)^2 + 
16 \big(x_3 - y_3 - \frac{1}{2} (y_1 x_{2}-y_{2}x_1)\big)^2 \Big)^{3/2}}.
\end{align*}
Evaluating this derivative over the plane $\Pi=\{y_1=0\}$ for $x_1>0$ we get 
\begin{equation}
\label{eq:X1y}
\begin{aligned}
X^y_{1} \Gamma(x,(0,y_2,y_3)) &=
\frac{ \big( x_1^2 + (x_2-y_2)^2\big) x_1 - 
4 (x_2-y_2)\big(y_3 - x_3 - \frac{1}{2} x_1 y_{2}\big)}
{ \pi \Big( \big( x_1^2 + (x_2-y_2)^2\big)^2 + 
16 \big(x_3 - y_3 + \frac{1}{2} y_{2}x_1)\big)^2 \Big)^{3/2}}\\
&=k_1(x,\hat y)+k(x,\hat y)
\end{aligned}
\end{equation}
where $k_1$ and $k$ are defined in \eqref{k1} and \eqref{k}.
Integrating \eqref{eq:X1y} over the plane $\Pi$ and assuming $y_1=0$, and $x_1>0$
we get \eqref{eq:DgPi}.
\end{proof}

\begin{remark}
\label{rk:int=1}
Notice that for each $r>0$ it holds 
\begin{equation}
\label{eq:intbound1}
\int_{\partial B_{r}(x)}  \escpr{\nabla_{\hh} \Gamma(x,y) , \nu(y)} d \sigma(y)=1.
\end{equation}
Indeed, by the mean value formula for each open subset $O\subset \hh^1$ such that $x \in O$, for each $r>0$ such that $B_r(x) \subset O$ and for each harmonic function $\psi \in \mathcal{H}(O)$ we have 
\[
\psi(x)= \int_{\partial B_{r}(x)} \psi(y) \escpr{\nabla_{\hh} \Gamma(x,y) , \nu(y)} d \sigma(y).
\] 
In particular if we consider $\psi \equiv1$ in $O$ we obtain \eqref{eq:intbound1}.
\end{remark}

\begin{lemma}\label{ga}
Let $x_0=(0,\hat x_0) \in \Pi$, $R>0$ and  $\hat B_R(\hat x_0)= \{\hat y \in \Pi \: \ \hat d (\hat x_0, \hat y)\leq R\} \subset \Pi$. Then the integral 
$$
\int_{ \hat B_R(\hat x_0)}  \escpr{\nabla_{\hh}^y \Gamma (x,(0, \hat y)) , X_1^y(y)} d \hat y
$$
is well defined if the first component $x_1$ of $x$ satisfies $x_1>0$ and tends to $1/2$ as $x\to x_0$.
\end{lemma}
\begin{proof}
Let $\{x^n\}_{n \in \nn}$ be a sequence of points in $\Omega=\{x_1>0\}$ converging  to $x_0$ as $n \to +\infty$ and $\eps_n>0$ small enough such that $B(x^n,\eps_n) \subset \Omega$ for each $n \in \nn$. Then we consider the bounded domain
$$\Omega_{n}^R= \{x_1 >0\} \cap B_R(x_0) \smallsetminus B(x^n,\eps_n).$$
By the divergence theorem  for each $n \in \nn $ we have 
\begin{equation}
\label{eq:2ballbound}
\begin{aligned}
0&= \int_{\Omega_{n}^R}  \Delta_{\hh} \Gamma (x^n, y) dy= \int_{\partial \Omega_n^R} \escpr{\nabla_{\hh}^y \Gamma(x_n,y), \nu_h(y)} d\sigma(y)\\
&= \int_{\partial B_R (x_0) \cap \{x_1 >0\}} \escpr{\nabla_{\hh}^y \Gamma(x_n,y), \nu (y)} d\sigma(y)\\
& \quad + \int_{\Pi \cap B_R(x_0)} \escpr{\nabla_{\hh}^y \Gamma(x_n,y), \nu (y)} d\sigma(y)\\
&\quad   -\int_{\partial B(x^n,\eps_n)} \escpr{\nabla_{\hh}^y \Gamma(x_n,y), \nu(y)} d\sigma(y).
\end{aligned}
\end{equation}
For each $n \in \nn$ the ball $B(x^n,\eps_n)$ is contained in $\{x_1>0\}$ thus by Remark \ref{rk:int=1} we get 
\[
\int_{\partial B(x^n,\eps_n)} \escpr{\nabla_{\hh}^y \Gamma(x_n,y), \nu_h(y)} d\sigma(y)=1.
\]
Noticing that $\Pi \cap B_R(x_0)= \hat B_R( \hat x_0)$ and rearranging terms in  \eqref{eq:2ballbound} we get
\[
\int_{\hat B_R( \hat x_0)} \escpr{\nabla_{\hh}^y \Gamma(x_n,y), \nu (y)} d\sigma(y)= 1- \int_{\partial B_R (x_0) \cap \{x_1 >0\}} \escpr{\nabla_{\hh}^y \Gamma(x_n,y), \nu (y)} d\sigma(y).
\]
Letting $n \to + \infty$ the left hand side of the previous equality converges to 
\[
1- \int_{\partial B_R (x_0) \cap \{x_1 >0\}} \escpr{\nabla_{\hh}^y \Gamma(x_0,y), \nu (y)} d\sigma(y)=\dfrac{1}{2},
\]
since we  only consider half of the integral equation \eqref{eq:intbound1}.
\end{proof}

The operator $K_1$ is totally degenerate while restricted 
to $\Pi$, so that we can not restrict it to functions defined on $\Pi$; however we 
can compute the limit from the interior of the set.

\begin{proposition}\label{frompositive}
Let  $g$ be  a Lipschitz compact supported function in 
$\Pi$ and $x_0$ be a point in $\Pi$.  
For $x \in \hh^1 \smallsetminus \Pi$ we consider
$$
K_1(g) (x) = \int_{\Pi} k_1(x,y) g(y) d\sigma(y).
$$
Then we  have
\begin{align*}
 &K_1(g) (x) \to \frac{1}{2} g(x_0) \quad \text{ as } x\to x_0^+,\\
 &K_1(g) (x) \to -\frac{1}{2} g(x_0) \quad  \text{ as } x\to x_0^-,
\end{align*}
so that $(K_1)^+=\tfrac{1}{2}\text{Id}$ while restricted to $\Pi$ and $(K_1)^-= - \frac{1}{2} \text{Id}$ while restricted to $\Pi$.
\end{proposition}

\begin{proof}
Let $R>0$ big enough such that  $\text{supp}(g)  \subset \hat B_R (\hat x_0)$.
Let us assume that $x=(x_1, \hat x)$, $x_1>0$ and 
\begin{align*}
K_1(g) (x) =& \int_{\Pi} k_1(x,y) g(y) d\sigma(y) =\int_{\hat B_R (\hat x_0)} k_1(x,y) (g(y) - g(x))d\sigma(y)\\
& + g(x) \int_{\hat B_R (\hat x_0)} k_1(x,y) d\sigma(y).
\end{align*}
On one hand we have
\begin{align*}
\left|\int_{\hat B_R (\hat x_0)} k_1(x,y) (g(y) - g(x)) d\sigma(y) \right|&\leq L \int_{\hat B_R (\hat x_0)} k_1(x,y) d(y, x)d\sigma(y)\\
&\leq L \int_{\hat B_R (\hat x_0)} \sqrt{x_1} d(y, x)^{-5/2}d\sigma(y) \to 0,
\end{align*}
$\text{  as }x \to x_0$ and where $L$ is the Lipschitz constant of $g$.
On the other hand by Lemma \ref{ga} we have  
\begin{align*}
&g(x) \int_{\hat B_R (\hat x_0)} k_1(x,y) d\sigma(y) =g(x)\int_{\hat B_R (\hat x_0)} (k_1(x,y)+ k(x,y) ) d\sigma(y) +\\
& \, -g(x) \int_{\hat B_R (\hat x_0)}  k(x,y) d\sigma(y) 
\xrightarrow[x \to x_0^+] {} \frac{1}{2} g(x_0)- g(x_0) \int_{\hat B_R (\hat x_0)}  k(x_0,y) d\sigma(y)=\frac{1}{2} g(x_0)
\end{align*}
by symmetry of the kernel $k$ restricted to $\Pi$, see Lemma \ref{lm:sypol2}.
Finally when $x_1<0$ the kernel  $k_1$ defined \eqref{k1} has the same sign of $x_1$ , then  $-k_1$ and $-x_1$ are positive and by Lemma \ref{ga}  we have 
\begin{align*}
&-g(x) \int_{\hat B_R (\hat x_0)} -k_1(x,y) d\sigma(y) =-g(x)\int_{\hat B_R (\hat x_0)} (-k_1(x,y)+ k(x,y) ) d\sigma(y) +\\
& \quad -g(x) \int_{\hat B_R (\hat x_0)}  k(x,y) d\sigma(y) 
\xrightarrow[(-x_1,x_2,x_3) \to x_0^+] {} -\frac{1}{2} g(x_0).
\qedhere
\end{align*}
\end{proof}

\begin{definition}
\label{def:kk}
As $x \to x_0^{\pm}$ the kernel $k(x,\hat y)$ defined in \eqref{k} converges to  the convolution kernel
\begin{equation}
\label{eq:kk}
k(\hat x- \hat y )=-
\dfrac{4}{\pi} \frac{  (x_2-y_2) \, (x_3 - y_3) }
{\Big( \big( x_2-y_2\big)^4 + 
16 \big(x_3 - y_3 \big)^2 \Big)^{3/2}}.
\end{equation}
Thus, if $g$ is a continuous compactly supported function in $\Pi$ the operator $K(g)$ converges to
\[
\int_{\Pi} k(\hat x- \hat y ) g(\hat y) d \sigma(y),
\]
that with an abuse of notation we also denoted by $K(g)$. 
\end{definition}

Hence the analogue of \cite[Theorem 4.4]{MR3600064} in this setting is the following
\begin{theorem}
Let $g$ be a Lipschitz compacty supported function in 
$\Pi$ and $x_0$ be a point in $\Pi$. Let $\mathcal{D}(g)$ be the double layer potential  defined in \eqref{eq:DgPi},  then the limits of $\mathcal{D}(g)(x)$ when $x$ tends to $x_0^+$ for $x \in \{x_1>0\}$ and when $x$ tends to $x_0^-$ for $x \in \{x_1<0\}$ exist. Moreover the limits verify  the following relations
\begin{align*}
 &\lim_{ x \to x_0^+} \mathcal{D}(g)(x)= \tfrac{1}{2} g(x_0)+ Kf(x_0) & \text{if} \quad x \in  \{x_1>0\}\\
& \lim_{ x \to x_0^-} \mathcal{D}(g)(x)= -\tfrac{1}{2} g(x_0)+ Kf(x_0) & \text{if} \quad x \in  \{x_1<0\},
\end{align*}
where $K$ is the operator with convolution kernel $k$ defined in \eqref{eq:kk}.
\end{theorem}

\begin{proof}
By Propositions \ref{frompositive} and Definition \ref{def:kk} we obtain 
$$\mathcal{D}(g)(x) \to (\frac{1}{2} I +K)(g)(x_0) $$  
in the limit from positive values of $x_1$, while
$$\mathcal{D}(g)(x) \to  (- \frac{1}{2} I +K)(g)(x_0) $$  
in the limit from negative values of $x_1$.
\end{proof}
\section{Invertibility of the double layer potential  on the intrinsic plane}
\label{sc:invert}
\subsection{The $C^{2,\alpha}$ estimates of $K$}
\label{sc:c2alphaestimete}
\begin{definition}[Classical H\"older class  $C^{1,\alpha}$]
Let $r \in \rr$, we say that a function $g$ defined on the boundary $\Pi_r=\{x=(r, x_2, x_3)\}$ is of class $C^{1, \alpha}(\Pi_r)$  
if and only if $\partial_2 g$ is a continuous function and there exists $C>0$ such that
\[
|\partial_{2}g(\hat y) - \partial_{2}g(\hat x)|  \le C \tilde{d}(\hat x,\hat y)^{\alpha}
\]
for each $\hat x=(x_2,x_3)$ and $\hat y=(y_2,y_3)$ in $ \Pi_r$  and where 
\begin{equation}
\label{eq:distancetilde}
\tilde d(\hat x,\hat y)=((x_2-y_2)^4+ 16 (x_3-y_3)^2)^{\tfrac{1}{4}}.
\end{equation}
\end{definition}
\begin{remark}
Notice that the induced distance
\[
\hat d(\hat x,\hat y)=((x_2-y_2)^4+ 16 (x_3-y_3 - \tfrac{1}{2} r (x_2-y_2))^2)^{\tfrac{1}{4}}.
\]
 on $\Pi_r$, considered in  Definition \ref{def:CKalpha},  is different from $\tilde{d}$. They coincide only when $r=0$, i.e. $\Pi_0=\Pi$.
\end{remark}
In addition, we set 
$$\|g\|_{1, \alpha} = \|g\|_{1} +\sup_{\hat x, \hat y \in \Pi_r} \frac{|\partial_{2}g(\hat y) - \partial_{2}g(\hat x)|}{\tilde d(\hat x,\hat y)^{\alpha}}$$
where 
$$\|g\|_{1}=\sup_{\hat x \in \Pi_r} g(\hat x)+ \sup_{\hat x \in \Pi_r} \partial_2 g(\hat x)$$

\begin{definition}[Classical H\"older classes  $C^{2,\alpha}$]
Let $r \in \rr$, we say that a function $g$ defined on the boundary $\Pi_r=\{x=(r, x_2, x_3)\}$ is of class $C^{2, \alpha}(\Pi_r)$  
if and only if $\partial^2_2 g$ and $\partial_3 g$ are continuous functions and  there exists $C>0$ such that
\[
|\partial^2_{2}g(\hat y) - \partial^2_{2}g(\hat x)|  \le C \tilde{d}(\hat x,\hat y)^{\alpha}
\]
and 
\[
|\partial_{3}g(\hat y) - \partial_{3}g(\hat x)|  \le C \tilde{d}(\hat x,\hat y)^{\alpha}
\]
for each $\hat x=(x_2,x_3)$ and $\hat y=(y_2,y_3)$ in $ \subset \Pi_r$.
In addition, we set 
$$\|g\|_{2, \alpha} =\|g\|_{2}+  [\partial_{3}g]_{\alpha}+[\partial_{2}^2g]_{\alpha} $$
where 
$$[\partial_{3}g]_{\alpha}=\sup_{\hat x, \hat y \in \Pi_r} \frac{|\partial_{3}g(\hat y) - \partial_{3}g(\hat x)|}{\tilde d(\hat x,\hat y)^{\alpha}},$$
$$[\partial_{2}^2 g]_{\alpha}=\sup_{\hat x, \hat y \in \Pi_r} \frac{|\partial_{2}^2g(\hat y) - \partial_{2}^2g(\hat x)|}{\tilde d(\hat x,\hat y)^{\alpha}} $$
and
$$\|g\|_{2} = \|g\|_{1}+ \sup_{\hat x \in \Pi_r} |\partial_2^2 g(\hat x)|+  \sup_{\hat x \in \Pi_r} |\partial_3 g(\hat x)|.$$
\end{definition}

\begin{proposition}
\label{Pr:C=Gamma}
A function $f$ belongs to $C^{2,\alpha}(\Pi_0)$ if and only if $f$ belongs to $\Gamma^{2,\alpha}(\Pi_0)$, namely for each $\hat x  \in \Pi_0$, $\rho>0$ there exists a polynomial $P_{\hat x}(\hat y)=a_\hx + b_\hx v_2+ c_\hx v_2^2+ d_\hx v_3$ with $\hv=\hy-\hx$ and $C>0$ such that 
\begin{equation}
\label{eq:CTE}
|f(\hy)-P_{\hx}(\hy)|<C\rho^{2+\alpha}
\end{equation}
for each $\hy \in B_{\rho}(\hx)$ (see Definition \ref{def:CKalpha}).
\end{proposition}
\begin{proof}
Assume that $f \in C^{2,\alpha}(\Pi_0)$. Let 
$$P_{\hat x}  (\hat y)= f(\hat x)+\partial_2 f(\hat x) (y_2-x_2)+\dfrac{ \partial_2^2 f(\hat x)}{2} (y_2-x_2)^2+ \partial_3 f(\hat x) (y_3-x_3).$$
By the Lagrange mean value theorem for the function $t\to f(y_2, x_3+t(y_3-x_3))$ with $t \in [0,1]$ we get 
\[
f(y_2,y_3)=f(y_2,x_3)+ \partial_3 f(\xi)(y_3-x_3)
\]
where $ \xi=(y_2,x_3+ \theta (y_3-x_3))$ for $\theta \in (0,1)$. Moreover, by the Taylor's formula with Lagrange remainder for the function $t\to f(x_2+t(y_2-x_2), x_3)$ with $t \in [0,1]$ we get 
\[
f(y_2,x_3)=f(x_2,x_3)+ \partial_2 f(\hat x)(y_2-x_2) + \dfrac{\partial_2^2 f(\eta)}{2}(y_2-x_2)^2
\]
where $ \eta=(x_2+\theta (y_2-x_2),x_3)$ for $\theta \in (0,1)$. Then we get 
\begin{align*}
f(y_2,y_3)&=f(x_2,x_3)+ \partial_2 f(\hat x)(y_2-x_2) + \dfrac{\partial_2^2 f(\eta)}{2}(y_2-x_2)^2+ \partial_3 f(\xi)(y_3-x_3)\\
&=P_{\hat x}  (\hat y)+\dfrac{\partial_2^2 f(\eta)-\partial_2^2 f(\hat x)}{2}(y_2-x_2)^2+  \partial_3 f(\xi)-\partial_3 f(\hat x)(y_3-x_3).
\end{align*}
Therefore 
\begin{align*}
|f(y_2,y_3)- P_{\hat x}  (\hat y)|&\le \dfrac{|\partial_2^2 f(\eta)-\partial_2^2 f(\hat x)|}{2}(y_2-x_2)^2+  |\partial_3 f(\xi)-\partial_3 f(\hat x)| \, |y_3-x_3|\\
&\le C \tilde{d}(\eta,\hat x)^{\alpha}  \tilde{d}(\hat y,\hat x)^2+  C\tilde{d}(\xi,\hat x)^{\alpha} \tilde{d}(\hat y,\hat x)^2 \le C \tilde{d}(\hat x, \hat y)^{2+\alpha}.
\end{align*}
Now, for any fixed $\hat x \in \Pi_0$ and $\rho > 0$, taking $\hat y \in   B_{\rho} (\hat x)$, clearly since $ \tilde{d}(\hat x, \hat y)^{2+\alpha} < \rho^{2+\alpha}$ we get 
\[
|f(\hat y)- P_{\hat x}  (\hat y)|< C \rho^{2+\alpha}.
\]
For the reverse implication we set 
\[
u_\rho(\hx)= \frac{u(\delta_{\rho} (\hx))}{\rho^2},
\]
where $\delta_\rho(\hat x)=(\rho x_2,\rho^2 x_3)$.
Let  $\hat x$, $\hat y$ two points at distance $\rho$ apart, by Remark \ref{rk:midpoint} there exists $\hat \xi$ such that  $\tilde{d}(\hx,\hat \xi),\tilde{d}(\hy,\hat \xi) < \frac{\sqrt{3}}{2} \rho $.
Then after a translation of $-\hat \xi$, we have $B_{\rho/2}=B_{\rho/2} (0) \subset B_{\sqrt{3}\rho}(\hat x), B_{\sqrt{3}\rho} (\hat y) $. Let  
\begin{align*}
\|P_{\hat x, \rho/2}- P_{\hat y, \rho/2} \|_{L^{\infty}(B_1)}&\le \|f_{\rho/2}-P_{\hat x, \rho/2}\|_{L^{\infty}(B_1)}+ \|f_{\rho/2}- P_{\hat y, \rho/2} \|_{L^{\infty}(B_1)}\\
&=\frac{4}{\rho^2} \sup_{\hat v \in B_{\rho/2}} |f(\hat v)- P_{\hat x}(\hat v)|+ \frac{4}{\rho^2} \sup_{\hat v \in B_{\rho/2}} |f(\hat v)- P_{\hat y}(\hat v)|\\
&\le \frac{4}{\rho^2} \sup_{\hat v \in B_{\sqrt{3}\rho}(\hat x)} |f(\hat v)- P_{\hat x}(\hat v)|+ \frac{4}{\rho^2} \sup_{\hat v \in B_{\sqrt{3} \rho}(\hat y)} |f(\hat v)- P_{\hat y}(\hat v)|\\
&\le 8 (3)^{1+\alpha/2} C \rho^{\alpha}.
\end{align*}
Notice that 
\[
(P_{\hat x, \rho/2} - P_{\hat y, \rho/2}) (\hat v)=\dfrac{4}{\rho^2}[ (a_x-a_y)+ (b_x-b_y)  \rho v_2 + (c_x- c_y) \rho^2 v_2^2 +(d_x-d_y) \rho^2 v_3].
\]
Then by Lemma \ref{lm:abcd} we get 
\begin{equation}
\label{eq:cdest}
\begin{aligned}
|a_\hx-a_\hy| \le 2 (3)^{1+\alpha/2} C \rho^{2+\alpha} \quad &\text{and} \quad |b_\hx-b_\hy| \le 4 (3)^{1+\alpha/2} C \rho^{1+\alpha}\\
|c_\hx-c_\hy| \le 4 (3)^{1+\alpha/2} C \rho^{\alpha} \quad &\text{and} \quad |d_\hx-d_\hy| \le 4 (3)^{1+\alpha/2} C \rho^{\alpha}.
\end{aligned}
\end{equation}
By assumption \eqref{eq:CTE} we easily get that $a_\hx=f(\hx)$, $f$ is continuous, $\partial_2 f(\hx)=b_\hx$, $\partial_3 f(\hx)=d_\hx$. Then by \eqref{eq:cdest} we obtain that $\partial_2 f,\partial_3 f$ are continuous and $\partial_3 f$ is $C^{\alpha}$ . Moreover, setting $e_2=(1,0)$  by  \eqref{eq:CTE} we have 
\[
(f(\hx+(h+s)e_2)- f(\hx+h e_2))-(f(\hx+s e_2)-f(\hx))=2c_\hx h s+O(s^{2+\alpha})+O(h^{2+\alpha}).
\]
Then there exists
\[
\lim_{h\to 0} \lim_{s\to 0} \frac{1}{h} \left( \frac{f(\hx+(h+s)e_2)- f(\hx+h e_2)}{s}-\frac{f(\hx+s e_2)-f(\hx)}{s} \right)=2c_\hx.
\]
On the other hand, letting $s\to 0$ in the previous limit we gain that
\[
\partial_2^2 f(\hx)=\lim_{h \to 0 } \frac{\partial_2 f(\hx+he_2)-\partial_2 f(\hx)}{h}=2 c_\hx.
\]
Finally, by \eqref{eq:cdest} we obtain that 
$|\partial_2^2 f(\hx)-\partial_2^2 f(\hy)|\le 8 (3)^{1+\alpha/2} C \tilde{d}(\hx,\hy)^{\alpha}$.
\end{proof}

\begin{remark}
\label{rk:midpoint}
Given two points $\hx,\hy \in \Pi$ such that $\rho=\tilde{d}(\hx,\hy)$ then there exists $\hat \xi=(\frac{x_2+y_2}{2}, \frac{x_3+3y_3}{4} )$ such that $\tilde{d}(\hat \xi,\hy)=\frac{\rho}{2} <\frac{\sqrt{3}}{2} \rho $ and $\tilde{d}(\hat \xi,\hx) <\frac{\sqrt{3}}{2} \rho$. Moreover if $\hx,\hy$ belongs to $B_R(\hx_0)$ for $\hx_0 \in \Pi$ then $\hat \xi$ in $B_{2R}(\hx_0)$.
\end{remark}

\begin{lemma}
\label{lm:abcd}
Let $\hat v \in \Pi$ and  $P(\hat v)=a+b v_2+c v_2^2+dv_3$. Assume that there exists $C>0$ such that  $\|P\|_{L^\infty(B_1)} \le C$, then $|a|\le C$ and $|b|,|c|,|d| \le 2C$
\begin{proof}
Setting $v_2=v_3=0$ we have $|a|\le C$. Let $\eps>0$, if $v_2=0$, $v_3=1/(1+\eps)$ we get $|a+ \frac{d}{1+\eps}| \le C$, thus $|d|\le 2C(1+\eps)$, letting $\eps \to 0 $ we get $|d|\le 2C$. Setting $v_2=\pm 1/(1+\eps)$, $v_3=0$ we obtain 
\[
\left|\frac{b}{1+\eps}+ \frac{c}{(1+\eps)^2}\right| \le 2C, \qquad \left|\frac{b}{1+\eps}-\frac{c}{(1+\eps)^2}\right| \le 2C. 
\]
Then we have 
\begin{align*}
\frac{|b|}{1+\eps}\le \frac{1}{2} \left( \left|\frac{b}{1+\eps}+ \frac{c}{(1+\eps)^2}\right|+ \left|\frac{b}{1+\eps}- \frac{c}{(1+\eps)^2}\right|\right) \le 2 C\\
\frac{|c|}{1+\eps}\le \frac{1}{2} \left( \left|\frac{b}{1+\eps}+ \frac{c}{(1+\eps)^2}\right|+ \left|\frac{b}{1+\eps}- \frac{c}{(1+\eps)^2}\right|\right) \le 2 C.
\end{align*}
Letting $\eps \to 0$ we get the desired inequalities.
\end{proof}
\end{lemma}

\begin{lemma}
\label{lm:sypol2}
Let $a_i \in \rr$ for each $i=1,\ldots,4$, $D_{\hat x} \subset \Pi$ be a set axially  symmetric with respect  to $y_2=x_2$ and  $y_3=x_3$ where $\hat x \in \Pi$. Let $k(\hat x- \hat y)$ be the convolution kernel given by 
\begin{equation}
\label{eq:convkernel}
k(\hat x- \hat y )=
\dfrac{4}{\pi} \frac{  (x_2-y_2) \, (x_3 - y_3) }
{\Big( \big( x_2-y_2\big)^4 + 
16 \big(x_3 - y_3) \big)^2 \Big)^{3/2}},
\end{equation}
then we have 
\begin{equation}
\label{eq:kp0}
\int_{D_{\hat x}} k(\hat x- \hat y) p (\hat x- \hat y) \, d \hat y=0,
\end{equation}
for each polynomial  
\[
p (\hat x- \hat y)=a_1+ a_2 (\hat x_2- \hat y_2)+a_3 (\hat x_3- \hat y_3)+ a_4(\hat x_2- \hat y_2)^2
\]
of degree less than or equal to $2$.
\end{lemma}
\begin{proof}
Changing the variable $\hat y_2'=\hat x_2-\hat y_2$ and $\hat y_3'=\hat x_3-\hat y_3$ we get that \eqref{eq:kp0} is equivalent to 
\begin{align*}
a_1 \int_{D_0} k(\hat y)  d \, \hat  y+ a_2   \int_{D_0} k(\hat y) \hat y_2  d \, \hat y +a_3 \int_{D_0} k(\hat y) \hat y_3  \,d \hat y + a_4 \int_{D_0} k(\hat y) \hat y_2^2 \,d \hat y=0,
\end{align*}
since the kernel $k$ is symmetric both in $y_2$ and in $y_3$. We denote by $D_0$ the translation of $D_{\hat x}$ in the origin.
\end{proof}

\begin{lemma}
\label{lm:coarea}
Let $0<\alpha<1$, $\tilde d(\hat x,\hat y)$ be the distance defined in \eqref{eq:distancetilde} and 
\[
 \hat B_r(\hat x)=\Big\{ (y_2,y_3) \in \Pi \ : \ \tilde d(\hat x,\hat y)<r\Big\}
\]
be the associated metric  ball of radius $r$ and center  $\hat x=(x_2,x_3)$. Let $k(\hat x- \hat y )$ be the convolution kernel defined in \eqref{eq:convkernel}. Then there exists a constant $C_3$  such that
\begin{equation}
\label{eq:intonball}
\int_{\hat B_r(\hat x)} |k(\hat x- \hat y )| \tilde{d} (\hat x,\hat y)^{j+\alpha} d \hat y\le C_{2} \, r^{j+\alpha}.
\end{equation}
When $j \neq 0$ then \eqref{eq:intonball} holds also for $\alpha=0$.
\end{lemma}
\begin{proof}
By Young's inequality there exist a constant $C_1$ such that 
\[
|x_2-y_2| |x_3-y_3|  \le \dfrac{|x_2-y_2|^3}{3} + \dfrac{2 \, |x_3-y_3| ^{\tfrac{3}{2}}}{3} \le C_1 (|x_2-y_2|^4+|x_3-y_3|^2)^{\tfrac{3}{4}},
\]
then we get that 
\[
|k(\hat x- \hat y )| \leq  C_1 \tilde d (\hat x , \hat y)^{-3}.
\]
Then we have 
\begin{equation*}
\begin{aligned}
\int_{B_r(\hat x)} |k(\hat x- \hat y )| \tilde{d} (\hat x,\hat y)^{j+\alpha} d \hat y &\leq  C_1 \int_{B_r(\hat x)}  \tilde{d} (\hat x,\hat y)^{j+\alpha-3} d \hat y\\
  &=  C_1 \int_{\hat B_r(0)}  \tilde {d} (0,(v_2,v_3))^{j+\alpha-3} d v_2 \, dv_3 \\
&= C_1  \int_0^r s^{j+\alpha-3}  \left( \int_{\partial \hat B_s(0)}  \dfrac{1}{|\nabla \tilde {d} |} \ d H^{1} \right) ds\\
&= 3 C_1 |\hat B_1(0)| \int_0^r s^{j+\alpha-1}  ds= C_2 \ r^{j+\alpha}
\end{aligned}
\end{equation*}
where $|\hat B_1(0)|$ is the $2$-Lebesgue measure of $\hat B_1(0)$ and $H^{1}$ is the $1$-dimensional Hausdorff measure.
In the second to last equality  we used 
\begin{equation}
\label{eq:surfacemeasureofball0}
3 r^{2}|\hat B_1(0)|=\int_{\partial \hat B_r(0)}  \dfrac{1}{|\nabla \tilde {d} |} \ d H^{1}.
\end{equation}
Indeed  by the coarea formula and using the induced dilations $\hat \delta_r (v_2, v_3)=(r \, v_2, r^2 \, v_3)$ we have
\[
r^{3}|\hat B_1(0)|=|\hat \delta_r(\hat B_1(0))|= |\hat B_r(0)|=\int_0^r \left( \int_{\partial \hat B_s(0)}  \dfrac{1}{|\nabla \tilde {d} |} \ d H^{1} \right) ds.
\]
Differentiating this last identity with respect to $r$ we obtain \eqref{eq:surfacemeasureofball0}.
\end{proof}

\begin{theorem}
\label{th:Kcont}
Let $k$ be the kernel defined in \eqref{eq:kk}, we set  
\[
K(f)( \hat x)=\int_{\Pi} k(\hat x-\hat y) f(\hat y) d \hat y
\]
for each $\hat x \in \Pi$. Assume that $f \in C^{2,\alpha}(\Pi)$ and $f$  compactly supported in $\Pi$ then there exists a constant $C$ such that 
\begin{equation}
\label{eq:C2ahe}
\| K (f) \|_{C^{2,\alpha}} \le C \| f \|_{C^{2,\alpha}}
\end{equation}
\end{theorem}

\begin{proof}

First of all notice that $k(\hat x,\hat y)$ is a convolution kernel, so that we can write $k(\hat x-\hat y)$ and by  Lemma \ref{lm:sypol2}, the kernel $k$ satisfies the cancellation condition, thus $K$ is a well-defined singular integral operator. Moreover, 
by Lemma \ref{lm:sypol2} we have 
\[
K(f)(\hat x)= \int_{\Pi \smallsetminus \hat B_1(\hat x)} k(\hat x-\hat y) f(\hat y) d \hat y+ \int_{ \hat B_1(\hat x)} k(\hat x-\hat y) (f(\hat y)- f(\hat x)) d \hat y,
\]
then, letting $B_R$ be a sufficiently large ball that contains the compact support of $f$ and using Lemma \ref{lm:coarea}, we get
\begin{align*}
|K(f)(\hat x)|&\le  \int_{\Pi \smallsetminus \hat B_1(\hat x)} |k(\hat x-\hat y)| |f(\hy)| d \hat y+ \| \partial_{2}f\|_{\infty} \int_{ \hat B_1(\hat x)} |k(\hat x-\hat y)| \tilde d(\hat x ,\hat y) d \hat y\\
&\le   C_1 \int_{\Pi \smallsetminus \hat B_1(\hat x)}  |f(\hy)| d \hat y+  \|f\|_{C^1} \le \tilde{C} \|f\|_{C^1}.
\end{align*}
Similarly we have 
\begin{align*}
\partial_{2} K(f)(\hx)&=\int_{B_R\smallsetminus B_{1}(\hx)} k(\hat x-\hat y) \partial_2 f(\hat y) d\hat y+ \int_{ B_{1} (\hx)} k(\hx-\hat y) (\partial_2 f(\hat y)  -\partial_2f(\hx) ) d\hat y,\\
\partial_{3} K(f)(\hx)&=\int_{B_R\smallsetminus B_{1}(\hx)} k(\hat x-\hat y) \partial_3 f(\hat y) d\hat y+ \int_{ B_{1} (\hx)} k(\hx-\hat y) (\partial_3 f(\hat y)  -\partial_3f(\hx) ) d\hat y,\\
\partial_{2}^2 K(f)(\hx)&=\int_{B_R\smallsetminus B_{1}(\hx)} k(\hat x-\hat y) \partial_2^2 f(\hat y) d\hat y+ \int_{ B_{1} (\hx)} k(\hx-\hat y) (\partial_2^2 f(\hat y)  -\partial_2^2f(\hx) ) d\hat y,
\end{align*}
then there exists a constant $\tilde{C}$ such that
\begin{align*}
|\partial_3K(f)(\hat x)|&\le \| \partial_3 f\|_{\infty} \int_{B_R \smallsetminus \hat B_1(\hat x)} |k(\hat x-\hat y)| d \hat y+ \| f\|_{C^{2,\alpha}} \int_{ \hat B_1(\hat x)} |k(\hat x-\hat y)| \tilde d(\hat x ,\hat y)^{\alpha} d \hat y\\
&\le \tilde{C} \|f\|_{C^{2,\alpha}},
\end{align*}
 $|\partial_2K(f)(\hat x)| \le \tilde{C} \|f\|_{C^{2}}$ and  $|\partial_2^2 K(f)(\hat x)| \le \tilde{C} \|f\|_{C^{2,\alpha}}$.
Therefore there exists a constant $\tilde{C}_2$ such that
\begin{equation}
\label{eq:C2e}
\|K(f)\|_{C^2} \le \tilde{C}_2  \| f \|_{C^{2,\alpha}}.
\end{equation}

Let $\hx_0$ be a point in $\Pi$. Fix $\hat x$ and $\hat z$ in $B_R=B_R(\hx_0)$ and let  $\delta=\tilde{d}(\hat x, \hat z)$. A direct computation shows  that 
$
k(\hx-\hy)= \frac{\partial}{\partial y_3} \ell (\hx-\hy)
$
where
\[
\ell (\hx-\hy)=\frac{1}{4\pi}  \frac{(x_2-y_2)}{\Big( \big( x_2-y_2\big)^4 + 
16 \big(x_3 - y_3 \big)^2 \Big)^{1/2}}.
\]
Notice that there exist a constant $C_3$ such that
\begin{equation}
\label{eq:dervker}
|\ell (\hx-\hy)|\le C_3 \tilde{d}(\hx,\hy)^{-1}, \,   |\partial_2\ell (\hx-\hy)|\le C_3 \tilde{d}(\hx,\hy)^{-2}, \,  |\partial_3\ell (\hx-\hy)|\le C_3 \tilde{d}(\hx,\hy)^{-2}
\end{equation}
Following  \cite[Lemma 4.4]{GT}, for each $\hx, \hz \in B_R(\hx_0)$ we have 
\[
\partial_2^2 K(f) (\hx)= \int_{B_{8R}} k(\hx-\hy) (\partial_2^2 f(\hy)-\partial_2^2 f(\hx)) d \hy + \partial_2^2 f(\hx)\int_{\partial B_{8R}} \ell (\hx-\hy) \nu_3 d \sigma(\hy),
\]
and 
\[
\partial_2^2 K(f) (\hz)= \int_{B_{8R}} k(\hz-\hy) (\partial_2^2 f(\hy)-\partial_2^2 f(\hz)) d \hy + \partial_2^2 f(\hz)\int_{\partial B_{8R}} \ell (\hz-\hy) \nu_3 d \sigma(\hy),
\]
where $\nu=(\nu_2,\nu_3)$ is the unit normal to $\partial B_{8R}$. Writing  $\delta=\tilde{d}(\hat x, \hat z)$ and letting  $\hat \xi$ be the point given by Remark \ref{rk:midpoint}  we obtain 
\begin{align*}
\partial_2^2 K(f) (\hx)-\partial_2^2 K(f) (\hz)=& \partial_2^2 f(\hx)I_1+(\partial_2^2 f(\hx)-\partial_2^2 f(\hz))I_2+I_3 + I_4\\
 &+(\partial_2^2 f(\hz)-\partial_2^2 f(\hx))I_5 + I_6,
\end{align*}
where 
\begin{align*}
I_1&= \int_{\partial B_{8R}} (\ell (\hx-\hy) -\ell (\hz-\hy) ) \nu_3 d \sigma(\hy)\\
I_2&= \int_{\partial B_{8R}}  \ell (\hz-\hy) \nu_3 d \sigma(\hy)\\
I_3&= \int_{ B_{\sqrt{3} \delta}(\hat \xi)}  k(\hx-\hy) (\partial_2^2 f(\hy)-\partial_2^2 f(\hx)) d \hy\\
I_4&=\int_{ B_{\sqrt{3}\delta}(\hat \xi)}  k(\hz-\hy) (\partial_2^2 f(\hz)-\partial_2^2 f(\hy)) d \hy\\
I_5&=\int_{ B_{8R}\smallsetminus B_{\sqrt{3}\delta}(\hat \xi)}  k(\hx,\hy) d \hy\\
I_6&=\int_{ B_{8R}\smallsetminus B_{\sqrt{3}\delta}(\hat \xi)} (k(\hx,\hy) - k(\hz,\hy) )(\partial_2^2 f(\hy)-\partial_2^2 f(\hz)) ) d \hy
\end{align*}
Let $\ga(t)=(r \sqrt{\cos(t)},\frac{r^2 \sin(t)}{4})$ for $t \in (-\frac{\pi}{2}, \frac{\pi}{2})$ be a parametrization of $\partial B_r$, where $\|v\|=\tilde{d}(0,v)$. Then we have 
\begin{equation}
\label{eq:partialBr}
H^1(\partial B_r)=2 \int_{-\frac{\pi}{2}}^{\frac{\pi}{2}} \| \dot{\ga}(t)\| dt= \frac{r}{2} \int_{-\frac{\pi}{2}}^{\frac{\pi}{2}} \frac{ \sqrt{\sin(t)^4+16 \cos(t)^4}}{\cos(t)} dt= r H^1(\partial B_1).
\end{equation}
Let  $\hat \eta=(x_2+\theta_1 (z_2-x_2),x_3) $ and $\hat \zeta=(z_2,x_3+ \theta_2 (z_3-x_3)) $ for $\theta_1, \theta_2 \in (0,1)$. Then  there exist constants $C_4,\ldots,C_9$ independent of $R,\delta$ such that 
\begin{align*} 
|I_1| \le& \tilde{d}(\hx,\hz)  \int_{\partial B_{2R}} |\partial_2\ell(\hat \eta-\hy)| d \sigma( \hy)+ \tilde{d}(\hx,\hz)^2 \int_{\partial B_{2R}} |\partial_3\ell(\hat \zeta-\hy)| d \sigma( \hy)\\
         \le  & C_4 \tilde{d}(\hx,\hz)^{\alpha} R^{-\alpha}=C_4 \left(\frac{\delta}{R}\right)^{\alpha} \qquad(\text{by equation} \, \eqref{eq:dervker} \, \text{and} \, \eqref{eq:partialBr} ). \\
|I_2| \le &C_4.\\
|I_3| \le &  \int_{ B_{\sqrt{3}\delta}(\hat \xi)}  |k(\hx-\hy)| |\partial_2^2 f(\hy)-\partial_2^2 f(\hx)| d \hy \le C_1 [\partial_2^2 f]_{\alpha} \int_{ B_{9\delta/2}(\hx)}  |k(\hx-\hy)|\tilde{d}(\hx, \hy)^{\alpha} \ dy.\\
&\le C_4  [\partial_2^2 f]_{\alpha} \delta^{\alpha}   \qquad \text{(by Lemma \ref{lm:coarea}) }\\
|I_4|  \le& C_4 [\partial_2^2 f]_{\alpha} \delta^{\alpha}  \qquad \text{as in the estimation of } \ I_3\\
|I_5|  =& \left| \int_{\partial (B_{8R}- B_{\sqrt{3}\delta}(\hat \xi))} \ell(\hx-\hy) \nu_3 d \sigma(\hy) \right|\\
\le&  \left| \int_{\partial B_{8R}} \ell(\hx-\hy) \nu_3 d \sigma(\hy) \right|+\left| \int_{\partial  B_{\sqrt{3} \delta}(\hat \xi)} \ell(\hx-\hy) \nu_3 d \sigma(\hy) \right| \le C_5,
\end{align*}
thanks to equations \eqref{eq:dervker} and\eqref{eq:partialBr}. Moreover, we have
\begin{align*} 
|I_6| \le & C_6 \delta \int_{ B_{8R}\smallsetminus B_{\sqrt{3}\delta}(\hat \xi)} |\partial_2 k(\hat \eta-\hy)| \, |(\partial_2^2 f(\hy)-\partial_2^2 f(\hz)) | d \hy\\
& \quad + \delta^2\int_{ B_{8R}\smallsetminus B_{\sqrt{3}\delta}(\hat \xi)} |\partial_3 k(\hat \zeta-\hy)| \, |(\partial_2^2 f(\hy)-\partial_2^2 f(\hz)) | d \hy\\
\le& C_7 [\partial_2^2 f]_{\alpha} \left( \delta \int_{\tilde{d}(\hy,\hat \xi)\ge \sqrt{3} \delta} \frac{\tilde{d} (\hy,\hz)^{\alpha}}{\tilde{d}(\hy, \hat \eta)^4} d \hy+ \delta^2 \int_{\tilde{d}(\hy,\hat \xi)\ge \sqrt{3} \delta} \frac{\tilde{d} (\hy,\hz)^{\alpha}}{\tilde{d}(\hy, \hat \zeta)^5} d \hy \right)\\
\le& C_8 [\partial_2^2 f]_{\alpha} \left( \delta \int_{\tilde{d}(\hy,\hat \xi)\ge \sqrt{3} \delta} \frac{\tilde{d} (\hy,\hat \xi)^{\alpha}}{\tilde{d}(\hy, \hat \xi)^4} d \hy+ \delta^2 \int_{\tilde{d}(\hy,\hat \xi)\ge \sqrt{3} \delta} \frac{\tilde{d} (\hy,\hat \xi)^{\alpha}}{\tilde{d}(\hy, \hat \xi)^5} d \hy \right)\\
&(\text{Since by Remark \ref{rk:midpoint}} \quad \tilde{d}(\hy,\hz) \le(1+\tfrac{\sqrt{3}}{2}) \tilde{d}(\hy,\hat \xi), (1- \tfrac{\sqrt{3}}{2}) \tilde{d}(\hat \xi,\hy) \le \tilde{d}(\hat \eta ,\hy)  \\
&   \qquad  \, \text{and} \ (1- \tfrac{\sqrt{3}}{2}) \tilde{d}(\hat \xi,\hy) \le \tilde{d}(\hy ,\hat \zeta) )\\
\le&  C_9 [\partial_2^2 f]_{\alpha} \left( \delta \int_{\sqrt{3}\delta}^{+\infty} s^{\alpha-2}  \ ds+ \delta^2 \int_{\sqrt{3}\delta}^{+\infty} s^{\alpha-3}  \ ds  \right)= C_9[\partial_2^2 f]_{\alpha} \delta^{\alpha},
\end{align*}
where in the last inequality we used the coarea formula and the equation \eqref{eq:surfacemeasureofball0}. Collecting terms we gain 
\begin{equation}
\label{eq:he22}
|\partial_2^2 K(f) (\hx)-\partial_2^2 K(f) (\hz)| \le C_{10} \left( R^{-\alpha} \|\partial_2^2  f\|_{\infty}+[\partial_2^2 f]_{\alpha} \right) \tilde{d}(\hx,\hz)^{\alpha}.
\end{equation}
Then choosing  $\hx_0=\hx$ and $R=1$ we have $B_R(\hx_0)=B_1(\hx)$, then  by equations \eqref{eq:he22}  we get 
\begin{equation}
\label{eq:he222}
\begin{aligned}
\frac{|\partial_2^2 K(f) (\hx)-\partial_2^2 K(f) (\hz)|}{\tilde{d}(\hx,\hz)^{\alpha}} &\le C_{11} (\|\partial_2^2  f\|_{\infty}+[\partial_2^2 f]_{\alpha})+ 2 \|\partial_2^2 K(f)\|_{\infty}\\
\text{(by eq.  \eqref{eq:C2e}) }& \le C_{11} (\|\partial_2^2  f\|_{\infty}+[\partial_2^2 f]_{\alpha})+ \tilde{C}_2 \| f\|_{C^{2,\alpha}}\\
&\le C_{12} \| f\|_{C^{2,\alpha}}
\end{aligned}
\end{equation}
Reasoning in the same way as above, we get 
\begin{equation}
\label{eq:he3}
\frac{|\partial_3 K(f) (\hx)-\partial_3 K(f) (\hz)|}{\tilde{d}(\hx,\hz)^{\alpha}} \le C_{13}  \| f\|_{C^{2,\alpha}}
 \end{equation}
Putting together equations \eqref{eq:C2e}, \eqref{eq:he222} and \eqref{eq:he3} we get the desired estimates \eqref{eq:C2ahe}. 
\end{proof}

\subsection{The reflection technique for singular integrals}
\begin{definition}
Let $g$ be a function on $\Pi$ and $x \in \hh^1 \smallsetminus \Pi$. We set 
$$
\tilde K_1(g) (x) = \int_{\Pi} \tilde k_1(x,\hat y) g(\hat y) d\sigma(\hat y), \qquad \tilde K(g) (x) = \int_{\Pi} \tilde k(x,\hat y) g(\hat y) d\sigma(\hat y)
$$
where
\begin{equation}
\label{eq:tidek1}
\tilde k_1(x,y)=
 \dfrac{1}{\pi} \frac{\big( x_1^2 + (x_2-y_2)^2\big) x_1 }
{\Big( \big( x_1^2 + (x_2-y_2)^2\big)^2 
+ 16 \big(x_3 - y_3 \big)^2 \Big)^{3/2}}
\end{equation}
and
\begin{equation}
\label{eq:tidek} \tilde k(x,y)=
 \dfrac{1}{\pi} \frac{(x_2-y_2)(x_3 - y_3  ) }
{\Big( \big( x_1^2 + (x_2-y_2)^2\big)^2 
+ 16 \big(x_3 - y_3 \big)^2 \Big)^{3/2}},
\end{equation}
\end{definition}

\begin{remark}
\label{rk:tiledek}
Notice that $\tilde k(x,y)$ defined in \eqref{eq:tidek} converges to the convolution kernel $k(\hat x-\hat y)$ defined in \eqref{eq:kk} as one approaches the boundary.
\end{remark}

\begin{lemma}\label{tildelimit}
Let  $g$ be  a Lipschitz compactly supported function in 
$\Pi$ and $x_0$ be a point in $\Pi$.  
For $x \in \hh^1 \smallsetminus \Pi$ we consider
$$
\tilde{K}_1(g) (x) = \int_{\Pi} \tilde{k}_1(x,y) g(y) d\sigma(y),
$$
where $\tilde{k}_1$ is defined in \eqref{eq:tidek1}.
Then we  have
\begin{align*}
 &\tilde{K}_1(g) (x) \to \frac{1}{2} g(x_0) \quad \text{ as } x\to x_0^+,\\
 &\tilde{K}_1(g) (x) \to -\frac{1}{2} g(x_0) \quad  \text{ as } x\to x_0^-,
\end{align*}
so that $(K_1)^+=\tfrac{1}{2}\text{Id}$ while restricted to $\Pi$ and $(K_1)^-= - \frac{1}{2} \text{Id}$ while restricted to $\Pi$.
 \end{lemma}
\begin{proof}  
By Proposition \ref{frompositive} $K_1 (g) (x)$ converges to $\pm \tfrac{1}{2} g$ and since 
\[
|K_1 (g) (x)- \tilde{K}_1(g)(x)| \leq \sup_{\hat x, \hat y} \left|1- \tfrac{\Big( \big( x_1^2 + (x_2-y_2)^2\big)^2 
+ 16 \big(x_3 - y_3 + \frac{1}{2}y_{2}x_1 \big)^2 \Big)^{3/2}}{\Big( \big( x_1^2 + (x_2-y_2)^2\big)^2 
+ 16 \big(x_3 - y_3 \big)^2 \Big)^{3/2}} \right| K_1(g)(x)
\]
we have $K_1 (g) (x)- \tilde{K}_1(g)(x)$ goes to zero when $x_1$ tends to $0$. Then also
$\tilde{K}_1(g)(x)$ converges to $\tfrac{1}{2} g$ when $x_1\to0^+$ and $\tilde{K}_1(g)(x)$ converges to $-\tfrac{1}{2} g$ when $x_1\to0^-$ .
\end{proof}

Given $r \in \rr$, let us denote $\Pi_r=\{x=(r, x_2, x_3)\}$. We consider the $C^{2, \alpha}(\Pi_r)$ norm with respect to the distance $\tilde{d}$  as we did in Section  \ref{sc:c2alphaestimete}. This choice allows us to completely decouple variables and we have

\begin{proposition}\label{IKnorm}  
Let $\Pi=\{ x_1=0\}$ and $K$ the singular operator defined by the kernel $k$, see \eqref{k}. Then we have
$$||(-\frac{1}{2}I + K)(g)||_{C^{2, \alpha}(\Pi)} = ||(\frac{1}{2}I + K)(g)||_{C^{2, \alpha}(\Pi)}.$$
\end{proposition}
\begin{proof}
Since the $C^{2,\alpha}$ norm on $\Pi_r$  with respect to the distance $\tilde{d}$ are independent on $r$, we have
$$||(\tilde K_1 + \tilde K)(g)(- \cdot, \cdot, \cdot)||_{C^{2, \alpha}(\Pi_{-r})} = ||(\tilde K_1 + \tilde K)(g)||_{C^{2, \alpha}(\Pi_r)}.$$
Letting $r$ to 0, and applying Lemma \ref{tildelimit} and Remark \ref{rk:tiledek} we get the thesis.
\end{proof}

\subsection{The method of continuity}
\label{sc:methodcontinuity}
Given $t \in [0,1]$ and $K$ the singular operator with kernel $k$ defined in \eqref{k}, we set 
\[
T_t=\frac{1}{2} I + t K.
\]
Notice that $T_1=\tfrac{1}{2}I + K$.
Let us consider the set
\[
A=\{ t \in [0,1] \ : \ T_t  \ \text{is invertible on} \ C^{2, \alpha}(\Pi)\}
\]
First of all we notice that $A \ne \emptyset$ since  $T_0=\frac{1}{2}I$ is invertible. By Theorem \ref{th:Kcont} we have that $T_t: C^{2, \alpha}(\Pi) \to C^{2, \alpha}(\Pi)$  is continuous, namely there exists a constant $C$ such that 
\begin{equation}
\label{eq:upperTt}
\| T_t(g)\|_{C^{2, \alpha}(\Pi)} \le C \|g\|_{C^{2, \alpha}(\Pi)}.
\end{equation}
\begin{proposition}
\label{pr:contmeth}
With the previous notations, it holds that
\begin{equation}
\label{eq:lowerTt}
\|g\|_{C^{2, \alpha}(\Pi)} \le 2  \| T_t(g)\|_{C^{2, \alpha}(\Pi)}
\end{equation}
\end{proposition}
\begin{proof}
Clearly we have
\[
g=\left(\frac{1}{2}I + tK\right) g + \left(\frac{1}{2}I - tK\right)g.
\]
Then  by Proposition \ref{IKnorm} we gain 
\begin{align*}
\|g \|_{C^{2,\alpha}(\Pi)} &\le \| (\tfrac{1}{2}I + t K) g\|_{C^{2,\alpha}(\Pi)} +  \| (\tfrac{1}{2}I - tK) g\|_{C^{2,\alpha}(\Pi)}\\
&= \| (\tfrac{1}{2}I + t K) g\|_{C^{2,\alpha}(\Pi)} +  \| (-\tfrac{1}{2}I + tK) g\|_{C^{2,\alpha}(\Pi)}\\
&= 2\| (\tfrac{1}{2}I + t K) g\|_{C^{2,\alpha}(\Pi)}.
\qedhere
\end{align*}
\end{proof}
By the estimates \eqref{eq:upperTt}, \eqref{eq:lowerTt} and the contraction mapping principle it follows that $A$ is both open and closed. Hence $A=[0,1]$ and $T_1=\tfrac{1}{2} I +  K $ is invertible from $C^{2,\alpha}(\Pi)$ to $C^{2,\alpha}(\Pi)$.

\section{The Poisson kernel and Schauder estimates}
\label{sc:Schest}

In this section we will show the Schauder estimates. First of all we consider the flat case as follows

\begin{theorem} \label{c:schauderGroups} 
Let $\Omega \subset {\mathbb{H}}^1$ be a bounded domain such that $\Omega \subset \{x_1>0\}$.
Let $\bar x\in \partial \Omega $ and assume that there exists  an open neighborhood $V$ of $\bar x $  such that  $  V \cap \partial \Omega \subset   \{x_1=0\}.$  Assume that  
$f \in C^\alpha(\bar \Omega)$ and $g \in  \Gamma^{2, \alpha} _0(\partial \Omega \cap V)$ and $0<\alpha<1$. Denote $u$ the unique solution to
$$\Delta_{\mathbb{H}} u=f\; \text{in}\  \Omega, \quad u= g \text{ on }\,  \partial \Omega.$$
Then
\begin{equation}
	\label{stime}
\|u\|_{C^{2, \alpha}(\bar \Omega)} \leq C (\|g\|_{ \Gamma^{2, \alpha} (\partial \Omega)} + \|f\|_{C^\alpha(\bar \Omega)}).\end{equation}
\end{theorem}

\begin{proof}
Let $\mathcal{D}$ be the double layer potential on $\Pi$ defined in \eqref{eq:DgPi}  and $K$ be  the operator  with convolution kernel $k(\hat x-\hat y)$ defined in \eqref{eq:kk}. Therefore we define
$$\mathcal{P}_0(g)(x) = \mathcal{D} ( (\tfrac{1}{2}I+K)^{-1}g) (x)=\int_{\Pi} ((\tfrac{1}{2}I+K)^{-1}g)(\hat y) (k_1(x,\hat y)+ k(x,\hat y)) d\hat y.$$
Then $u=\mathcal{P}_0(g)$ satisfies 
\begin{equation}
\begin{cases}
\Delta_{\mathbb{H}} u=0 & \text{in} \quad \Omega \\
u=g & \text{on} \quad  \{x_1=0\},
\end{cases}
\label{P}
\end{equation}
since $\Omega \subset \{x_1>0\}$.
For each $g \in  \Gamma^{2, \alpha} _0(\partial \Omega \cap V)$ we set $$K_{\partial \Omega} (g)=I(g)- \mathcal{P}_0(g).$$
If we choose $V_0 $ such that $supp (g) \subset \subset V_0\subset \subset V$, then 
$K_{\partial \Omega} (g)(x) = 0 $ for every $x \in \partial \Omega \cap V$.
On the other side  $g(x)=0$ for $ x \in \partial \Omega \smallsetminus V_0$, thus we have
\begin{align*}
K_{\partial \Omega} (g)( x) &= \mathcal{P}_0(g) (x)=  \mathcal{D} ( (\tfrac{1}{2}I+K)^{-1}g) (x)\\
&=\int_{\Pi} ((\tfrac{1}{2}I+K)^{-1}g)(\hat y) (k_1(x,\hat y)+ k(x,\hat y)) d\hat y.
\end{align*}
then the kernel defining $K_{\partial \Omega}$ has no singularities when  
$ x \in\partial \Omega \smallsetminus \{x_1=0\}$.
As a consequence $K_{\partial \Omega}$ is a compact operator, and $I-K_{\partial \Omega}$ can be 
explicitly inverted. 
Then we set  
\begin{equation}
\label{eq:Pg}
\mathcal{P}(g) = \mathcal{P}_0((I-K_{\partial \Omega})^{-1} g)= \mathcal{D} ( (\tfrac{1}{2}I+K)^{-1}(I-K_{\partial \Omega})^{-1}g)
\end{equation}
then $\mathcal{P}(g)$ is a Poisson kernel such that $u=\mathcal{P}(g)$ satisfies 
\begin{equation}
\begin{cases}
\Delta_{\mathbb{H}} u=0 & \text{in} \quad \Omega \\
u=g & \text{on} \quad  \partial \Omega.
\end{cases}
\label{P}
\end{equation}
Since  $\mathcal{D}$ is continuous from $C^{2,\alpha}(\Pi)$ into $C^{2,\alpha}(\{x_1>0\})$ (see for instance \cite[Main Lemma 13.12]{GreinerStein} or \cite{NagelStein79}),   $(\tfrac{1}{2}I+K)^{-1}$ is continuous from  $C^{2,\alpha}(\Pi)$ into $C^{2,\alpha}(\Pi)$ thanks to Section \ref{sc:methodcontinuity}, {$C^{2,\alpha}(\Pi)$ coincides with $\Gamma^{2,\alpha}(\Pi)$ by Proposition \ref{Pr:C=Gamma} and $(I-K_{\partial \Omega})^{-1}$ is continuous from $\Gamma^{2,\alpha}(\partial \Omega)$ into $\Gamma^{2,\alpha}(\Pi)$}  we get  $u=\mathcal{P}(g)$ defined in \eqref{eq:Pg} verifies 
\[
\| u\|_{C^{2,\alpha}(\Omega)}\le C \|g\|_{\Gamma^{2,\alpha}(\partial \Omega)}.
\]
On the other hand by the interior estimates (see for instance \cite{MR1135924}) we have that the solution $v$ of 
\begin{equation}
\begin{cases}
\Delta_{\mathbb{H}} v=f & \text{in} \quad \Omega \\
v=0 & \text{on} \quad  \partial \Omega.
\end{cases}
\label{L}
\end{equation}
verifies 
\[
\| v\|_{C^{2,\alpha}(\Omega)}\le \tilde{C} \|f\|_{C^{\alpha}( \Omega)}.
\]
Hence considering  the function $u+v$ instead of $u$  we get the thesis.
\end{proof}

Therefore we have solved the problem assuming that the boundary is locally a plane.  Now we have to flatten the boundary and extend the result to general boundaries.

\begin{theorem} \label{c:schauder2} 
Let $\Omega \subset {\mathbb{H}}^1$ be a bounded domain and $u$ is the unique
solution to
$$\Delta_{\mathbb{H}} u=f\; \text{in}\  \Omega, \quad u= g \text{ on }\,  \partial \Omega, $$
where $f \in C^\alpha(\bar \Omega)$ and $g \in \Gamma^{2, \alpha} (\partial \Omega)$ and $0<\alpha<1$. 
Let $\bar x\in  \partial\Omega $ be a non-charateristic point, $V\subset \hh^1$ be an open neighborhood of $\bar x $ without charateristic points and $\phi\in C^\infty_0(V)$ be a bump function equal to $1$ in neighborhood $V_0 \subset \subset V$ of $\hat x$. Then we have 
$\phi u \in C^{2, \alpha}(\bar \Omega\cap V)$ and 
\begin{equation}
	\label{stime}
\|\phi u\|_{C^{2, \alpha}(\bar \Omega\cap V)} \leq C (\|g\|_{\Gamma^{2, \alpha} (\partial \Omega)} + \|f\|_{C^\alpha(\bar \Omega)}).\end{equation}
\end{theorem}

\begin{proof}

Let us denote by $\Omega$ a smooth, open bounded set in $\hh^1$ and let 
$0\in \partial \Omega$ be a non characteristic point.  
The boundary of $\Omega$ can be identified in a neighborhood $V$ with the graph of a regular 
function $w$, defined on a neighborhood $\hat V=V\cap \rr^{2}$ of $0$: 
$$\partial \Omega \cap V= \{(w(\hat s), \hat s): \hat s\in \hat V  \}.$$
We can as well assume that $w(0) = 0$,  $\nabla w=0$. This implies that 
\begin{equation}\label{tuttoqui}
w(\hat s) = O(|\hat s|^2) 
\end{equation}
as $\hat s \to 0$. 
On the set $V$ the function 
$\Xi(s_1, \hat s) = (s_1 - w(\hat s), \hat s) $
is a diffeomorphism. It sends  $\partial \Omega\cap V$ to a subset of the plane $\{x_1 =0\}$:
$$\Xi(\partial \Omega \cap V) =\{(x_1, \hat x): x_1 =0\}= \Pi_{\Xi}.$$
Moreover, we have 
\begin{equation}\label{ohscusa}\Delta_\Xi =  d\Xi(\Delta_{\hh^1}), \end{equation}
with fundamental solution 
$$ \Gamma_\Xi(x) = \Gamma(x_1 + w(\hat x), \hat x).$$
For $x_1$ small enough we have
$$ \Gamma_\Xi(x_1, \hat x) = \Gamma(x_1+w(\hat x), \hat x)
= \Gamma(x_1, \hat x)
 +  R(x_1, \hat x), $$
 where 
 $$  R(x_1,\hat x)=w(\hat x) \nabla \Gamma( x_1 +t w(\hat x), \hat x) $$
 for some $t \in (0,1)$. Furthermore we have that 
 $$X_{1,\Xi}^x=d\Xi(X_1^x)=\partial_{x_1}-\tfrac{x_2}{2} \partial_{x_3} w(\hat x) \partial_{x_1}-\tfrac{x_2}{2} \partial_{x_3}.$$
 Notice that $\Gamma$ is a rational function that goes as $d^{-Q+2}$, its first derivatives go as $d^{-Q+1}$ and its second derivatives go as $d^{-Q}$. On the other side the function $w(\hat x)$ has a 0 of order 2 thus $w(\hat x)$ goes as $d^{2}$. Then we have 
 \[
 X_{1,\Xi}^y \Gamma_\Xi(0, \hat y)=X_1^y \Gamma (0,\hat y)+ \hat R(0,\hat y),
 \]
 where 
 \[
 \hat R(0,\hat y)=X_1^y R(0,\hat y)-\tfrac{y_2}{2} \partial_{y_3} w(\hat y) \partial_{y_1} R(0,\hat y) -\tfrac{y_2}{2} \partial_{y_3} w(\hat y) \partial_{y_1} \Gamma(0,\hat y)
 \]
 that goes as 
 \[
 |\hat R(0,\hat x)| \le \hat d^{-Q+2}(0,\hat x),
 \]
 where $\hat d$ is the induce distance.
Therefore the operator  $K_{\hat R}$ with kernel $\hat R$ is compact since the homogenous dimension of the boundary is $Q-1$. Therefore also thanks to the left-invariance of the distance and of the fundamental solution  we get that the double layer potential 
\[
\mathcal{D}(\phi g)(x)= \int _{\Pi_{\Xi}} X_{1,\Xi}^y \Gamma_\Xi(x, \hat y) ( \phi g)(\hat y) d \hat y= K_1(\phi g)(x)+K(\phi g)(x)+K_{\hat R}(\phi g)(x)
\]
converges to $(\frac{1}{2} I +K + K_{\hat R})(\phi g)(x_0) $ in the limit $x \to x_0^+$ from positive values of $x_1$.  In the previous equation $K_1$ and $K$ are the operator defined in Proposition \ref{prop:K1Kplane}, $x_0=(0,\hat x_0) \in \hat V$. With an abuse of notation we denote in the same way the function $\phi g$ compactly supported in $V \cap \partial \Omega$ and the function $\phi g \circ \Xi^{-1}$ compactly supported on $\Pi_{\Xi}$. Since $\frac{1}{2} I +K$  is invertible, $$(\frac{1}{2} I +K + K_{\hat R})(x_0) =(\frac{1}{2} I +K )(I   +  (\frac{1}{2} I +K )^{-1}K_{\hat R})(x_0) $$ 
Since $K_{\hat R}$ is compact and $(\frac{1}{2} I +K )^{-1}$ is bounded, 
then $(\frac{1}{2} I +K )^{-1}K_{\hat R}$ is  compact, so that 
$(I   +  (\frac{1}{2} I +K )^{-1}K_{\hat R})$ is invertible, thus $\frac{1}{2} I +K + K_{\hat R}$ is invertible. Therefore we define
$$\mathcal{P}(\phi g)(x) = \mathcal{D} ( (\tfrac{1}{2}I+K+K_{\hat R})^{-1} \phi g) (x).$$
Then $u=\mathcal{P}(\phi g)$ satisfies 
\begin{equation}
\begin{cases}
\Delta_{\Xi} u=0 & \text{in} \quad \{x_1>0\} \\
u=\phi g & \text{on} \quad  \{x_1=0\}.
\end{cases}
\label{PP}
\end{equation}
In particular, since  $\phi=1$ on $V_0$ we have that $\phi u=\phi \mathcal{P}(\phi g)$ solves 
\begin{equation*}
\begin{cases}
\Delta_{\Xi} (\phi u) =0 & \text{in} \quad \Xi^{-1}(V_0 \cap \Omega) \\
\phi u=\phi g & \text{on} \quad  \Xi^{-1}(V_0 \cap \partial \Omega),
\end{cases}
\end{equation*}
thus changing variables and noticing that $\phi=1$ on $V_0$ we gain that $\mathcal{P}(\phi g)\circ \Xi$ solves 
\begin{equation}
\begin{cases}
\Delta_{\hh^1}u=0 & \text{in} \quad V_0 \cap \Omega \\
 u= g & \text{on} \quad  V_0 \cap \partial \Omega.
\end{cases}
\label{PP}
\end{equation} 
Since  $\mathcal{D}$ is continuous from $C^{2,\alpha}(\Pi_{\Xi})$ into $C^{2,\alpha}(\{x_1>0\})$ (see for example \cite[Main Lemma 13.12]{GreinerStein} or \cite{NagelStein79}),   $(\tfrac{1}{2}I+K+K_{\hat R})^{-1}$ is continuous form  $C^{2,\alpha}(\Pi_{\Xi})$ into $C^{2,\alpha}(\Pi_{\Xi})$, { $C^{2,\alpha}(\Pi_{\Xi})$ coincides with $\Gamma^{2,\alpha}(\Pi_{\Xi})$ by Proposition \ref{Pr:C=Gamma} } and $\Xi$ is a smooth diffeomorphism  we get  that $\phi u=\phi \mathcal{P}(\phi g)\circ \Xi$ verifies 
\[
\| \phi u\|_{C^{2,\alpha}(V \cap 	\bar{\Omega})}\le C \|\phi g\|_{\Gamma^{2,\alpha}(\partial \Omega)} \le C \|g\|_{\Gamma^{2,\alpha}(\partial \Omega)}.
\]
On the other hand by the interior estimates (see for instance \cite{MR1135924}) we have that the solution $v$ of 
\begin{equation}
\begin{cases}
\Delta_{\mathbb{H}} v=f & \text{in} \quad \Omega \\
v=0 & \text{on} \quad  \partial \Omega.
\end{cases}
\label{L}
\end{equation}
verifies 
\[
\| v\|_{C^{2,\alpha}(\bar{\Omega})} \le \tilde{C} \|f\|_{C^{\alpha}( \bar{\Omega})}.
\]
Hence considering  the function $u+v$ instead of $u$  we get the thesis.
\end{proof}

\begin{corollary} \label{c:schauder3} 
Let $\Omega \subset {\mathbb{H}}^1$ be a bounded domain without characteristic points 
and let $u$ is the unique
solution to
$$\Delta_{\mathbb{H}} u=f\; \text{in}\  \Omega, \quad u= g \text{ on }\,  \partial \Omega, $$
where $f \in C^\alpha(\bar \Omega)$ and $g \in \Gamma^{2, \alpha} (\partial \Omega)$ and $0<\alpha<1$. 
We have 
\begin{equation}
	\label{stime}
\|u\|_{C^{2, \alpha}(\bar \Omega)} \leq C (\|g\|_{\Gamma^{2, \alpha} (\partial \Omega)} + \|f\|_{C^\alpha(\bar \Omega)}).\end{equation}
\end{corollary}

\begin{proof}
We can cover the boundary by a finite number of balls $\{B_i\}_{i=1,\ldots,N}$ and an associated partition of the unity $\phi_1, \ldots,\phi_N$. Then on each ball $B_i$ we have that $\phi_i g$ is compactly supported in $B_i$, then by Theorem 
\ref{c:schauder2} we get 
\[
\| \phi_i u\|_{C^{2,\alpha}(B_i \cap \bar{\Omega})} \le C \|g\|_{C^{2,\alpha}(\partial \Omega)}.
\]
Since we have 
\[
\| u \|_{C^{2,\alpha} (\bar{\Omega})}= \| u \|_{C^{2,\alpha} (\Omega \smallsetminus \medcup_{i=1}^N B_i )} +  \| \sum_{i=1}^N \phi_i u\|_{C^{2,\alpha}(\bar{\Omega})},
\]
we estimate the first term by means of the interior estimates and the second therm as follows
\[
\| \sum_{i=1}^N \phi_i u\|_{C^{2,\alpha}(\bar{\Omega})}\le \sum_{i=1}^N \| \phi_i u\|_{C^{2,\alpha}(B_i \cap \bar{\Omega})} \le N C \|g\|_{\Gamma^{2,\alpha}(\partial \Omega)}.
\]
Hence we get the result.
\end{proof}

\section{Generalization to Heisenberg-type groups}
\label{sc:GHT}
\begin{definition}
A Heisenberg-type algebra is a finite-dimensional real Lie algebra $\mathfrak{g}$ which can be endowed with an inner product $\escpr{\cdot,\cdot}$ such that
\[
[\frak{z}^{\perp},\frak{z}^{\perp}]= \frak{z},
\]
where $\frak{z}$ is the center $\mathfrak{g}$ and moreover, for every fixed $z \in \frak{z}$  the map 
\[
J_z: \frak{z}^{\perp} \to \frak{z}^{\perp}
\]
defined by 
\[
\escpr{J_z(v),w}=\escpr{z,[v,w]} \quad \forall v,w \in \frak{z}^{\perp}
\]
 is an orthogonal map whenever $\escpr{z,z}=1$.
\end{definition}
A Heisenberg-type group $\mathbb{G}$ ($H$-type group, in short) is a connected and simply
connected Lie group whose Lie algebra is an $H$-type algebra.

Let $n,m \in \nn$, $m\ge2$ and $n\ge 1$.
Following \cite[Chapter 18]{BLU} a prototype of $H$-type group $(\rr^{m+n},\delta_{\lambda}, \circ )$  is given by $\rr^{m+n}$ equipped with the group  law
\[
(x,t) \circ (y,\tau)=\left( \begin{array}{cc} x_k+ y_k  & k=1,\ldots,m\\ t_k+\tau_k +\frac{1}{2}\escpr{A^{(k)} x,y } & k=1,\ldots,n
  \end{array} \right)
\]
and the dilation $\delta_{\lambda}(x,t)=(\lambda x, \lambda^2 t)$. Here $A^{(k)}$ is a skew-symmetric orthogonal matrix, such that, 
\begin{equation}
\label{eq:hyA}
A^{(k)} A^{(\ell)}+ A^{(\ell)} A^{(k)}=0,
\end{equation}
for $k=1,\ldots,n$ with $k \ne \ell$. By \cite[Theorem 18.2.1]{BLU} any  $H$-type group is naturally isomorphic to a prototype $H$-group. Therefore we use the notation $\mathbb{G}$ for the prototype $H$-type group $(\rr^{m+n},\delta_{\lambda}, \circ )$ associated to a $H$-type group.  A family of left invariant vector fields that agree with $\tfrac{\partial}{\partial x_j}$ for $j=1,\ldots,m$ at the origin is given by
\[
X_j=\frac{\partial}{\partial x_j}+ \frac{1}{2} \sum_{k=1}^n \left( \sum_{i=1}^m a_{j,i}^{k} x_i \right) \frac{\partial}{\partial t_k}
\]
Setting that $m= \dim(\mathfrak{z}^{\perp})$ and $n=\dim(\mathfrak{z})$ we have  that $\{ X_1,\ldots, X_m \}$ is a basis of the horizontal distribution $\mathfrak{z}^{\perp}$. Then that  we have
\[
[X_i,X_j]= \sum_{k=1}^n a^{k}_{j,i} \frac{\partial}{\partial t_k}
\]
and setting $Z_k= \frac{\partial}{\partial t_k}$ for $k=1,\ldots,n$ we get that $Z_1,\ldots,Z_n$ is an orthonormal basis of  $\mathfrak{z}$. 
 The homogenous dimension $Q$  is given by $Q=m+2n$. We denote by $\nabla_{\mathbb{G}}$ the horizontal gradient 
$
\nabla_{\mathbb{G}}= (X_1,\ldots,X_m)
$
and by $\nabla$ the standard Euclidean gradient.  The sub-Laplacian operator is given by
\[
\Delta_{\mathbb{G}}=  \sum_{k=1}^{m} X_i^2= \divv_{\mathbb{G}}(\nabla_{\mathbb{G}} ),
\]
where $\divv_{\mathbb{G}}(\phi)= X_1(\phi_1)+\ldots+ X_m(\phi_m)$ for $\phi=\phi_1 X_1+\ldots+ \phi_m X_m \in \mathfrak{z}^{\perp}$. Is is well known (see \cite[Chapter 5]{BLU}) that the sub-Laplacian admits a unique fundamental solution $\hat{\Gamma} \in C^{\infty}(\rr^{m+n} \smallsetminus \{0\})$, $\hat{\Gamma} \in L_{\text{loc}}^1(\rr^{m+n})$, $\hat{\Gamma}(x,t) \to 0$ when  $\xi=(x,t)$  tends to infinity and such that 
\[
\int_{\rr^{m+n}} \hat{\Gamma}(x,t) \,  \Delta_{\mathbb{G}}   \varphi(x,t) \,  dx \ dt = -\varphi(0) \quad \forall \varphi \in C^{\infty}(\rr^{m+n}).
\]
\begin{definition}
We call Gauge norm on $\mathbb{G}$ a homogeneous symmetric norm $d$ smooth out of the origin and satisfying
\[
\Delta_{\mathbb{G}} (d(x,t)^{2-Q})=0 \quad \forall (x,t) \ne (0,0) .
\]
\end{definition}
Following \cite{BLU} a Gauge norm in $\mathbb{G}$ is given by 
\[
|\xi|_{\GG}=\left( | x |^4+ 16 |t|^2\right)^{\tfrac{1}{4}},
\]
where $\xi=(x,t)$.
Therefore there exists a positive constant $C_Q$ such that  
\[
\hat{\Gamma}(\xi)=C_Q \, |\xi|_{\GG}^{2-Q}= \tfrac{C_Q}{ \Big( |x|^4 + 
16 |t|^2 \Big)^{\frac{(2-Q)}{4}}}.
\]
Strictly speaking $| \cdot|_{\GG}$ and $\hat \Gamma$ are defined on the algebra $\mathfrak{g}$ and $d(\xi,\eta)=|\eta^{-1}\circ \xi|_{\GG}$ on the group $\mathbb{G}$. Indeed, for every couple of points $\xi=(x,t)$, $\eta=(y,\tau)$ in $\rr^{m+n}$ there exists and are unique coefficients  $v=(v_1,\ldots,v_m)$ and $z=(z_1,\ldots,z_n)$ such that
$$\xi = \exp\left(\sum_{i=1}^m v_i X_i + \sum_{k=1}^n z_k Z_k \right) (\eta).$$
 We call these  coefficients  $(v,z) = \text{Log}_{\eta}(\xi)$ and a straightforward computation shows that $ \text{Log}_{\eta}(\xi)= \eta^{-1} \circ \xi$. Finally we define the fundamental solution  $\Gamma(\xi, \eta)$ on $\mathbb{G}$ as $\hat \Gamma(\text{Log}_{\eta}(\xi))= \hat \Gamma(\eta^{-1} \circ \xi)$, that is given by 
\begin{equation}
\label{eq:fundsolHT}
\Gamma(\xi, \eta)=
C(Q)  \Big( \big| x-y \big|^4 + 
16 \sum_{k=1}^n \big(t_k - \tau_k - \frac{1}{2}  \escpr{A^{(k)}y,x} \big)^2 \Big)^{\frac{(2-Q)}{4}}.
\end{equation}

\begin{example}
\label{eq:nothormander}
Let us consider the $H$-type group given by $\rr^{5}$ with the following vector fields
\[
X_1=\frac{\partial }{\partial x_1} - \frac{x_2}{2} \frac{\partial }{\partial t_1}, \quad X_2=\frac{\partial }{\partial x_2} +\frac{x_1}{2} \frac{\partial }{\partial t_1}, \quad  X_3=\frac{\partial }{\partial x_3} +x_1 \frac{\partial }{\partial t_2},
\]
that generate the horizontal distribution $\mathfrak{z}^{\perp}$,
and  $Z_1=\tfrac{\partial }{\partial t_1}$, $ Z_2=\frac{\partial }{\partial t_2}$. When we consider $\Pi=\{x_1=0\}$ we obtain that $\mathfrak{z}^{\perp} \cap \Pi$, that is generated by $\tfrac{\partial }{\partial x_2}$ and $\tfrac{\partial }{\partial x_3}$,  does not satisfy the H\"{o}rmander condition. In particular this a Carnot group, different from $\hh^1$, that does not satisfy the structure condition $(1.5)$ in \cite{BCC19}.

\end{example}

\begin{proposition}
\label{prop:K1KplaneHT}
Let $\Omega=\{(x,t) \in \rr^{m+n} \ : \ x_1>0\} \subset \mathbb{G}$ and $\partial  \Omega =\{x_1=0\}= \Pi$.  Then the double layer potential
$ \mathcal{D}(g)(x)$ is given by 
\begin{equation}
\label{eq:DgPiHT}
\mathcal{D}(g)(\xi)= K_1(g)(\xi) + K(g)(\xi)
\end{equation}
for $\xi \in \Omega$, where $K_1$ and $K$ are operators with kernels 
respectively $k_1$ and $k$ defined as
\begin{equation}\label{k1H}k_1(\xi ,\hat \eta)=C_Q (Q-2) \dfrac{ |x-(0,\hat y)|^2 x_1  }
{  \Big( \big| x-(0,\hat y) \big|^4 + 
16 \sum_{k=1}^n \big(\tau_k - t_k - \frac{1}{2}  \escpr{ A^{(k)}x,(0,\hat y)} \big)^2 \Big)^{\frac{Q+2}{4}}}
\end{equation}
\begin{equation}\label{kH}k(\xi,\hat \eta)=C_Q (2-Q) \dfrac{ 4 \sum_{i=2}^m \sum_{k=1}^n  \big(\tau_k - t_k - \frac{1}{2}  \escpr{A^{(k)}x,(0,\hat y)} \big) a_{1,i}^k (y_i-x_i)}
{  \Big( \big| x-(0,\hat y) \big|^4 + 
16 \sum_{k=1}^n \big(\tau_k - t_k - \frac{1}{2}  \escpr{ A^{(k)}x,(0,\hat y)} \big)^2 \Big)^{\frac{Q+2}{4}}},
\end{equation}
where $\hat y=(y_2,\ldots, y_m)$ and $\hat{\eta}=(0,\hat y,\tau) \in \Pi$.
\end{proposition}
\begin{proof}
First of all we have 
\[
X_1 \hat \Gamma (v,z)= C_Q (2-Q) \dfrac{ |v|^2 v_1  + 4 \sum_{i=1}^m \sum_{k=1}^n  a_{1,i}^k v_i z_k}
{  \Big( | v |^4 + 
16 |z|^2 \Big)^{\frac{Q+2}{4}}}
\]

Then, by left invariance an explicit computation shows that the derivative \eqref{eq:fundsolHT} with respect to $X^{\xi}_1$ is given by
\begin{align*}
&X^{\xi}_{1}(\Gamma(\xi, \eta)) = (X_{v_1} \hat \Gamma)(\eta^{-1} \circ \xi) =\\
&=
C_Q (2-Q) \dfrac{ |x-y|^2 (x_1-y_1)  + 4 \sum_{i=1}^m \sum_{k=1}^n  \big(t_k - \tau_k - \frac{1}{2}  \escpr{A^{(k)}y,x} \big) a_{1,i}^k (x_i-y_i)}
{  \Big( \big| x-y \big|^4 + 
16 \sum_{k=1}^n \big(t_k - \tau_k - \frac{1}{2}  \escpr{A^{(k)}y,x} \big)^2 \Big)^{\frac{Q+2}{4}}}
\end{align*}
Since $\Gamma$ is symmetric we also have 
\begin{align*}
&X^{\eta}_{1} \Gamma(\xi,\eta) =\escpr{\nabla_{\GG}^{\eta} \Gamma(\xi,\eta), X_1^{\eta}}\\
&=
C_Q (2-Q) \dfrac{ |x-y|^2 (y_1-x_1)  + 4 \sum_{i=1}^m \sum_{k=1}^n  \big(\tau_k - t_k - \frac{1}{2}  \escpr{A^{(k)}x,y} \big) a_{1,i}^k (y_i-x_i)}
{  \Big( \big| x-y \big|^4 + 
16 \sum_{k=1}^n \big(t_k - \tau_k - \frac{1}{2}  \escpr{A^{(k)}y,x} \big)^2 \Big)^{\frac{Q+2}{4}}}\\
&=
C_Q (2-Q) \dfrac{ |x-y|^2 (y_1-x_1)  + 4 \sum_{i=1}^m \sum_{k=1}^n  \big(\tau_k - t_k - \frac{1}{2}  \escpr{A^{(k)}x,y} \big) a_{1,i}^k (y_i-x_i)}
{  \Big( \big| x-y \big|^4 + 
16 \sum_{k=1}^n \big(\tau_k - t_k - \frac{1}{2}  \escpr{ A^{(k)}x,y} \big)^2 \Big)^{\frac{Q+2}{4}}}\\
\end{align*}
Evaluating this derivative over the plane $\Pi=\{y_1=0\}$ for $x_1>0$ and notincing that $a_{1,i}^k=0$ we get
\begin{align*}
X^{\eta}_{1} \Gamma(\xi,\eta) &= C_Q (2-Q) \dfrac{ - |x- (0,\hat y)|^2 x_1  + 4 \sum_{i=2}^m \sum_{k=1}^n  \big(\tau_k - t_k - \frac{1}{2}  \escpr{A^{(k)}x,(0,\hat y)} \big) a_{1,i}^k (y_i-x_i)}
{  \Big( \big| x-(0,\hat y) \big|^4 + 
16 \sum_{k=1}^n \big(\tau_k - t_k - \frac{1}{2}  \escpr{ A^{(k)}x,(0,\hat y)} \big)^2 \Big)^{\frac{Q+2}{4}}}\\
&=k_1(\xi,\hat \eta)+k(\xi,\hat \eta)
\end{align*}
\end{proof}

\begin{remark}
\label{rk:int=1}
Notice that for each $r>0$ it holds 
\begin{equation}
\label{eq:intbound1HT}
\int_{\partial B_{r}(\xi)}  \escpr{\nabla_{\GG} \Gamma(\xi,\eta) , \nu(\eta)} d \sigma(\eta)=1.
\end{equation}
Indeed, by the mean value formula for each open subset $O\subset \GG$ such that $\xi \in O$, for each $r>0$ such that $B_r(\xi) \subset O$ and for each harmonic function $\psi \in \mathcal{H}(O)$ we have 
\[
\psi(\xi)= \int_{\partial B_{r}(\xi)} \psi(\eta) \escpr{\nabla_{\GG} \Gamma(\xi,\eta) , \nu(\eta)} d \sigma(\eta).
\] 
In particular if we consider $\psi \equiv1$ in $O$ we obtain \eqref{eq:intbound1HT}.
\end{remark}

Let $\Omega=\{x_1>0\} \subset \GG$ and $\partial  \Omega =\{x_1=0\}= \Pi$. Then the induced distance $ \hat d$ is given by 
\begin{equation}
 \hat d(\hat \xi,\hat \eta)=|(0,\hat \eta)^{-1} \circ (0,\hat \xi) |_{\GG}
\end{equation}
for each $\hat \xi =(0,\hat x,t)$ and $\hat \eta =(0,\hat y,\tau)$ in $\Pi$ and  the induced ball is given by 
\[
\hat B_r(\hat \xi)=\Big\{ \hat \eta \in \Pi \ : \  \hat d(\hat \xi,\hat \eta)<r\Big\}.
\]

\begin{lemma}\label{gaHT}
Let $\xi_0=(0,\hat x_0,t_0) \in \Pi$, $R>0$ and  $\hat B_R(\hat \xi_0)= \{\hat \eta \in \Pi \: \ \hat d (\hat \xi_0, \hat y)\leq R\} \subset \Pi$. Then the integral 
$$
\int_{ \hat B_R(\hat \xi_0)}  \escpr{\nabla_{\GG}^{\eta} \Gamma (\xi,(0, \hat \eta)) , X_1^{\eta}(\hat{\eta})} d \hat \eta
$$
is well defined if the first component $x_1$ of $\xi$ satisfies $x_1>0$ and tends to $1/2$ as $\xi \to \xi_0$.
\end{lemma}
\begin{proof}
Let $\{\xi^n\}_{n \in \nn}$ be a sequence of points in $\Omega=\{x_1>0\}$ converging  to $\xi_0$ as $n \to +\infty$ and $\eps_n>0$ small enough such that $B(\xi^n,\eps_n) \subset \Omega$ for each $n \in \nn$. Then we consider the bounded domain
$$\Omega_{n}^R= \{x_1 >0\} \cap B_R(\xi_0) \smallsetminus B(\xi^n,\eps_n).$$
By the divergence theorem  for each $n \in \nn $ we have 
\begin{equation}
\label{eq:2ballboundHT}
\begin{aligned}
0&= \int_{\Omega_{n}^R}  \Delta_{\GG} \Gamma (\xi^n, \eta) d\eta= \int_{\partial \Omega_n^R} \escpr{\nabla_{\GG}^{\eta} \Gamma(\xi_n,\eta), \nu(\eta)} d\sigma(\eta)\\
&= \int_{\partial B_R (\xi_0) \cap \{x_1 >0\}} \escpr{\nabla_{\GG}^\eta \Gamma(\xi_n,\eta), \nu (\eta)} d\sigma(\eta)\\
& \quad + \int_{\Pi \cap B_R(\xi_0)} \escpr{\nabla_{\GG}^\eta \Gamma(\xi_n,\eta), \nu (\eta)} d\sigma(\eta)\\
&\quad   -\int_{\partial B(\xi^n,\eps_n)} \escpr{\nabla_{\GG}^\eta \Gamma(\xi_n,\eta), \nu(\eta)} d\sigma(\eta).
\end{aligned}
\end{equation}
For each $n \in \nn$ the ball $B(\xi^n,\eps_n)$ is contained in $\{x_1>0\}$ thus by Remark \ref{rk:int=1} we get 
\[
\int_{\partial B(\xi^n,\eps_n)} \escpr{\nabla_{\GG}^y \Gamma(\xi_n,\eta), \nu_h(\eta)} d\sigma(\eta)=1.
\]
Noticing that $\Pi \cap B_R(x_0)= \hat B_R( \hat x_0)$ and rearranging terms in  \eqref{eq:2ballboundHT} we get
\[
\int_{\hat B_R( \hat \xi_0)} \escpr{\nabla_{\GG}^{\eta} \Gamma(\xi_n,\eta), \nu (\eta)} d\sigma(\eta)= 1- \int_{\partial B_R (\xi_0) \cap \{x_1 >0\}} \escpr{\nabla_{\GG}^\eta \Gamma(\xi_n,\eta), \nu (\eta)} d\sigma(\eta).
\]
Letting $n \to + \infty$ the left hand side of the previous equality converges to 
\[
1- \int_{\partial B_R (\xi_0) \cap \{x_1 >0\}} \escpr{\nabla_{\GG}^\eta \Gamma(\xi_0,\eta), \nu (\eta)} d\sigma(\eta)=\dfrac{1}{2},
\]
since we  only consider half of the integral equation \eqref{eq:intbound1HT}.
\end{proof}

The operator $K_1$ is totally degenerate while restricted 
to $\Pi$, so that we can not restrict it to functions defined on $\Pi$; however we 
can compute the limit from the interior of the set.

\begin{proposition}\label{frompositiveH}
Let  $g$ be  a Lipschitz compact supported function in 
$\Pi$ and $\xi_0$ be a point in $\Pi$.  
For $\xi \in \GG \smallsetminus \Pi$ we consider
$$
K_1(g) (\xi) = \int_{\Pi} k_1(\xi,\eta) g(\eta) d\sigma(\eta).
$$
Then we  have
\begin{align*}
 &K_1(g) (\xi) \to \frac{1}{2} g(\xi_0) \quad \text{ as } \xi \to \xi_0^+,\\
 &K_1(g) (\xi) \to -\frac{1}{2} g(\xi_0) \quad  \text{ as } \xi \to \xi_0^-,
\end{align*}
so that $(K_1)^+=\tfrac{1}{2}\text{Id}$ while restricted to $\Pi$ and $(K_1)^-= - \frac{1}{2} \text{Id}$ while restricted to $\Pi$.
\end{proposition}

\begin{proof}
Let $R>0$ big enough such that  $\text{supp}(g)  \subset \hat B_R (\hat \xi_0)$.
Let us assume that $\xi=(x_1, \hat x, t)$, $x_1>0$ and 
\begin{align*}
K_1(g) (\xi) =& \int_{\Pi} k_1(\xi,\eta) g(\eta) d\sigma(\eta) =\int_{\hat B_R (\hat \xi_0)} k_1(\xi,\eta) (g(\eta) - g(\xi))d\sigma(\eta)\\
& + g(\xi) \int_{\hat B_R (\hat \xi_0)} k_1(\xi,\eta) d\sigma(\eta).
\end{align*}
On one hand we have
\begin{align*}
\left|\int_{\hat B_R (\hat \xi_0)} k_1(\xi,\eta) (g(\eta) - g(\xi)) d\sigma(\eta) \right|&\leq L \int_{\hat B_R (\hat \xi_0)} k_1(\xi,\eta) d(\xi,\eta)d\sigma(\eta)\\
&\leq L \int_{\hat B_R (\hat \xi_0)} \sqrt{x_1} d(\eta, \xi)^{-Q+1 +\frac{1}{2}}d\sigma(\eta) \to 0,
\end{align*}
$\text{  as } \xi \to \xi_0$ and where $L$ is the Lipschitz constant of $g$.
On the other hand by Lemma \ref{gaHT} we have  
\begin{align*}
&g(\xi) \int_{\hat B_R (\hat \xi_0)} k_1(\xi,\eta) d\sigma(\eta) =g(\xi)\int_{\hat B_R (\hat \xi_0)} (k_1(\xi,\eta)+ k(\xi,\eta) ) d\sigma(\eta) +\\
& \, -g(\xi) \int_{\hat B_R (\hat \xi_0)}  k(\xi,\eta) d\sigma(\eta) 
\xrightarrow[\xi \to \xi_0^+] {} \frac{1}{2} g(\xi_0)- g(\xi_0) \int_{\hat B_R (\hat \xi_0)}  k(\xi_0,\eta) d\sigma(\eta)=\frac{1}{2} g(\xi_0)
\end{align*}
by symmetry of the kernel $k$ restricted to $\Pi$, see Lemma \ref{lm:sypol2H}.
Finally when $x_1<0$ the kernel  $k_1$ defined \eqref{k1H} has the same sign of $x_1$ , then  $-k_1$ and $-x_1$ are positive and by Lemma \ref{gaHT}  we have 
\begin{align*}
&-g(\xi) \int_{\hat B_R (\hat \xi_0)} -k_1(\xi,\eta) d\sigma(\eta) =-g(\xi)\int_{\hat B_R (\hat \xi_0)} (-k_1(\xi,\eta)+ k(\xi,\eta) ) d\sigma(\eta) +\\
& \quad -g(\xi) \int_{\hat B_R (\hat \xi_0)}  k(\xi,\eta) d\sigma(\eta) 
\xrightarrow[(-x_1,\hat x,t) \to \xi_0^+] {} -\frac{1}{2} g(\xi_0).
\qedhere
\end{align*}
\end{proof}

\begin{definition}
\label{def:kkH}
As $\xi \to \xi_0^{\pm}$ the kernel $k(\xi,\hat \eta)$ defined in \eqref{kH} converges to  the kernel
\begin{equation}
\label{eq:kkH}
k(\hat \xi, \hat \eta )=C_Q (2-Q) \dfrac{ 4 \sum_{i=2}^m \sum_{k=1}^n  \big(\tau_k - t_k - \frac{1}{2}  \escpr{\hat{A}^{(k)} \hat x,\hat y} \big) a_{1,i}^k (y_i-x_i)}
{  \Big( \big| \hat x- \hat y \big|^4 + 
16 \sum_{k=1}^n \big(\tau_k - t_k - \frac{1}{2}  \escpr{ \hat A^{(k)} \hat x,\hat y} \big)^2 \Big)^{\frac{Q+2}{4}}},
\end{equation}
where $\hat A^{(k)}=(a_{ij}^k)_{i,j=2,\ldots,m}$. Notice  that 
$k(\hat \xi, \hat \eta )= \hat k ((\hat v,z))$ where $(0,\hat v, z)=\text{Log}_{(0,\hat \xi)}((0,\hat \eta))$, $\hat v=(v_2,\ldots,v_m)$ and 
\begin{equation}
\label{eq:kkHO}
\hat k((\hat v, z))= C_Q (2-Q) \dfrac{ 4 \sum_{i=2}^m \sum_{k=1}^n  a_{1,i}^k v_i z_k}
{  \Big( | \hat v |^4 + 
16 |z|^2 \Big)^{\frac{Q+2}{4}} }
\end{equation}
Thus, if $g$ is a continuous compactly supported function in $\Pi$ the operator $K(g)$ converges to
\[
\int_{\Pi} k(\hat \xi, \hat \eta ) g(\hat \eta) d \sigma(\eta),
\]
that with an abuse of notation we also denoted by $K(g)$. 
\end{definition}

Hence the analogous of \cite[Theorem 4.4]{MR3600064} in this setting is the following
\begin{theorem}
Let $g$ be a Lipschitz compacty supported function in 
$\Pi$ and $\xi_0$ be a point in $\Pi$. Let $\mathcal{D}(g)$ be the double layer potential  defined in \eqref{eq:DgPiHT},  then the limits of $\mathcal{D}(g)(\xi)$ when $\xi$ tends to $\xi_0^+$ for $\xi \in \{x_1>0\}$ and when $\xi$ tends to $\xi_0^-$ for $\xi \in \{x_1<0\}$ exist. Moreover the limits verify  the following relations
\begin{align*}
 &\lim_{ \xi \to \xi_0^+} \mathcal{D}(g)(\xi)= \tfrac{1}{2} g(\xi_0)+ Kf(\xi_0) & \text{if} \quad \xi \in  \{x_1>0\}\\
& \lim_{ \xi \to \xi_0^-} \mathcal{D}(g)(\xi)= -\tfrac{1}{2} g(\xi_0)+ Kf(\xi_0) & \text{if} \quad \xi \in  \{x_1<0\},
\end{align*}
where $K$ is the operator with convolution kernel $k$ defined in \eqref{eq:kkH}.
\end{theorem}

\begin{proof}
By Propositions \ref{frompositiveH} and Definition \ref{def:kkH} we obtain 
$$\mathcal{D}(g)(\xi) \to (\frac{1}{2} I +K)(g)(\xi_0) $$  
in the limit from positive values of $x_1$, while
$$\mathcal{D}(g)(\xi) \to  (- \frac{1}{2} I +K)(g)(\xi_0) $$  
in the limit from negative values of $x_1$.
\end{proof}

\subsection{Invertibility of the double layer potential  on the intrinsic plane}
\label{sc:invertH}
\subsection{The $C^{2,\alpha}$ estimates of $K$}
\label{sc:c2alphaestimeteH}
Let $r\in \rr$ and $\Pi_r=\{\xi=(r, \hat{x},t)\}$. Let  $\hat \xi=(\hat x, t), \hat \eta=(\hat y, \tau) \in \Pi_r$ and $\hat{v}=(v_2,\ldots,v_m)$ and $z=(z_1,\ldots,z_n)$ such that  $(0,\hat v,z) = \text{Log}_{(0,\hat \eta)}((0,\hat \xi))$. We set $\hat X_j= X_j \big|_{x_1=0}$. 
On $\Pi_r$ we consider the distance 
\begin{equation}
\label{eq:distancetildeH}
\tilde{d}(\hat \xi,\hat \eta)= |(0,-\eta)\circ(0,\hat \xi) |_{\GG}=\left( | \hat v |^4+ 16 |z|^2\right)^{\tfrac{1}{4}}
\end{equation}
instead of $\hat d$ on $\Pi_r$. Notice that $\tilde d$ coincides with $\hat d$ on $\Pi_r$ if and only if $r=0$.

\begin{definition}[Classical H\"older classes  $C^{2,\alpha}$]
Let $\hat{X}^2 g (\hxi)$ be the \emph{horizontal tangential Hessian} given by
\[
\hat{X}^2 g(\hxi)_{i,j}=\frac{\hat X_i \hat X_j g(\hxi )+\hat X_j \hat X_i g(\hxi)}{2}, \qquad i,j=2,\ldots,m.
\]
Let $r \in \rr$, we say that a function $g$ defined on the boundary $\Pi_r=\{x=(r, \hx, t)\}$ is of class $C^{2, \alpha}(\Pi_r)$  
if and only if $\hat X_{i}  \hat X_j g$ for $i=2,\ldots,m$ and $Z_k g$ for $k=1,\ldots,n$ are continuous functions and  there exists $C>0$ such that
\[
|\hat{X}^2 g(\he)_{i,j}  - \hat{X}^2 g(\hxi)_{i,j} |  \le C \tilde{d}(\hat \eta,\hat \xi)^{\alpha}
\]
for $i,j=2,\ldots,m$ and 
\[
|Z_k g(\hat \eta) - Z_k g(\hat \xi)|  \le C \tilde{d}(\hat \xi,\hat \eta)^{\alpha},
\]
for $k=1,\ldots,n$ and for each $\hat \xi, \hat \eta$ in $ \Pi_r$.
In addition, we set 
$$\|g\|_{2, \alpha} =\|g\|_{2}+ \max_{i,j=2,\ldots,m}[(\hat{X}^2 g)_{i,j}]_{\alpha}+\max_{k=1,\ldots,n}[Z_k g]_{\alpha} $$
where 
$$[Z_k g]_{\alpha}=\sup_{\hat \xi, \hat \eta \in \Pi_r} \frac{|Z_k g(\hat \eta) - Z_k g(\hat \xi)| }{\tilde d(\hat \xi,\hat \eta)^{\alpha}},$$
$$[(\hat{X}^2 g)_{i,j}]_{\alpha}=\sup_{\hat \xi, \hat \eta \in \Pi_r} \frac{|\hat{X}^2 g(\he)_{i,j}  - \hat{X}^2 g(\hxi)_{i,j}| }{\tilde d(\hat \xi,\hat \eta)^{\alpha}} $$
and
$$\|g\|_{2} = \|g\|_{\infty}+ \max_{i=2,\ldots,m} \sup_{\hat \xi \in \Pi_r}| \hat X_i g(\hat \xi)| + \max_{i,j=2,\ldots,m} \sup_{\hat \xi \in \Pi_r} |\hat X^2 g(\hat \xi)_{i,j}|+  \max_{k=1,\ldots,n} \sup_{\hat \xi \in \Pi_r}| Z_k g(\hat \xi)|.$$
\end{definition}

\begin{proposition}
\label{prop:Gamma=C2H}
A function $f$ belongs to $C^{2,\alpha}(\Pi_0)$ if and only if $f$ belongs to $\Gamma^{2,\alpha}(\Pi_0)$, namely for each $\hat \xi  \in \Pi_0$, $\rho>0$ there exists a polynomial $P_{\hat \xi}(\hat  \eta)=a_\hxi + b_\hxi \cdot \hat v+ \hat v^T C_\hxi \hat v+ d_\hxi  \cdot z$ with $(0,\hat v,z) = \text{Log}_{(0,\hat \eta)}((0,\hat \xi))$ and $C>0$ such that 
\begin{equation}
\label{eq:CTEH}
|f(\he)-P_{\hxi}(\he)|<C\rho^{2+\alpha}
\end{equation}
for each $\he \in B_{\rho}(\hxi)$ (see Definition \ref{def:CKalpha}).
\end{proposition}
\begin{proof}
Assume that $f \in C^{2,\alpha}(\Pi_0)$. Let 
\begin{align*}
P_{\hxi}  (\he)&= f(\hxi)+\hat X f(\hxi) \cdot (\hy-\hx)+\tfrac{1}{2}   (\hy-\hx)^T \, \hat{X}^2 f(\hxi)\, \cdot (\hy-\hx) \\
& \quad +\sum_{k=1}^n Z_k f(\hxi) (\tau_k-t_k-\tfrac{1}{2} \escpr{\hat A^{(k)} \hat x, \hat y}),
\end{align*}
where  $\hat X=(\hat X_2,\ldots,\hat X_m)$, $Z=(Z_1,\ldots,Z_n)$  and 
\[
\hat{X}^2 f(\hxi)_{i,j}=\frac{\hat X_i \hat X_j f(\hxi )+\hat X_j \hat X_i f(\hxi)}{2}, \qquad i,j=2,\ldots,m
\]
By the Taylor's formula with Lagrange remainder for the function $s\to f(\gamma(s))$ with $\dot{\gamma}(s)=\sum_{i=2}^m v_i \hat X_i $, with $\hat v=(\hy-\hx)$ $s \in [0,1]$, $\gamma(0)=(\hx,t)$, $\gamma(1)=(\hy, \bar{t})$ we get 
\[
f(\hy,\bar{t})=f(\hx,t)+ \hat X f(\hxi)\cdot(\hy-\hx) + \frac{1}{2} (\hy-\hx)^T\hat X^2 f(\hat \mu) \cdot (\hy-\hx)
\]
where $ \hat \mu=\gamma(\theta)$ for $\theta \in (0,1)$ and $\bar{t}_k=t_k+\frac{1}{2} \escpr{\hat A^{(k)} \hx, \hy}$.
Moreover, by the Lagrange mean value theorem for the function $s\to f(\hy, \bar{t}+s(\tau-\bar{t}))$ with $s \in [0,1]$ we get 
\[
f(\hy, \tau)=f(\hy,\bar{t})+ Z f( \hat \zeta) \cdot (\tau-\bar{t})
\]
where $ \hat \zeta=(\hy, \bar{t}+ \theta (\tau-\bar{t}))$ for $\theta \in (0,1)$. 
 Then we get 
\begin{align*}
f(\hy, \tau)&=f(\hx,t)+  \hat X f(\hxi)\cdot(\hy-\hx) + \tfrac{1}{2} (\hy-\hx)^T\hat X^2 f(\mu) \cdot (\hy-\hx)\\
& \quad+  \sum_{k=1}^n Z_k f( \hat \zeta)  (\tau_k-t_k-\tfrac{1}{2} \escpr{\hat A^{(k)}\hat x, \hat y})\\
&=P_{\hxi}  (\he)+\tfrac{1}{2}(\hy-\hx)^T [\hat X^2 f(\mu)-\hat X^2(\hxi)](\hy-\hx) \\
& \quad +  \sum_{k=1}^n [Z_k f(\hat \zeta)-Z_k f(\hxi)](\tau_k-t_k-\tfrac{1}{2} \escpr{\hat A^{(k)}\hat x, \hat y}).
\end{align*}
Therefore 
\begin{align*}
|f(\he)- P_{\hxi}  (\he)|&\le \sup_{i,j=2,\ldots,m } |\hat X^2 f(\hat \mu)_{i,j}- \hat X^2 f(\hxi)_{i,j}| |\hat v|^2+  \sup_{k=1,\ldots,n}|Z_k f(\hat \zeta)-Z_k f(\hxi)| \, |z|\\
&\le C \tilde{d}(\hat \mu,\hxi)^{\alpha}  \tilde{d}(\he,\hxi)^2+  C\tilde{d}(\xi,\hat \zeta)^{\alpha} \tilde{d}(\he,\hxi)^2  \le \tilde{C} \tilde{d}(\he, \hxi)^{2+\alpha}.
\end{align*}
Now, for any fixed $\hxi \in \Pi_0$ and $\rho > 0$, taking $\he \in   B_{\rho} (\hxi)$, clearly since $ \tilde{d}(\hxi, \he)^{2+\alpha} < \rho^{2+\alpha}$ we get 
\[
|f(\he)- P_{\hxi}  (\he)|< C \rho^{2+\alpha}.
\]
For the reverse implication we set 
\[
u_\rho(\hxi)= \frac{u(\delta_{\rho} (\hxi))}{\rho^2},
\]
where $\delta_\rho(\hxi)=(\rho \hx ,\rho^2 t)$.
Let  $\hxi$, $\he$ two points at distance $\rho$ apart, by Remark \ref{rk:midpointH} there exists $\hat \zeta$ such that  $\tilde{d}(\hxi,\hat \zeta),\tilde{d}(\he,\hat \zeta) < \frac{\sqrt{3}}{2} \rho $.
Then after a translation by the group low  of $-\hat \zeta$, we have $B_{\rho/2}=B_{\rho/2} (0) \subset B_{\sqrt{3}\rho}(\hxi), B_{\sqrt{3}\rho} (\he) $. Let  
\begin{align*}
&\|P_{\hxi, \rho/2}- P_{\he, \rho/2} \|_{L^{\infty}(B_1)}\\
&\le \|f_{\rho/2}-P_{\hxi, \rho/2}\|_{L^{\infty}(B_1)}+ \|f_{\rho/2}- P_{\he, \rho/2} \|_{L^{\infty}(B_1)}\\
&=\frac{4}{\rho^2} \sup_{(\hat v,z) \in B_{\rho/2}} |f(\hat v,z)- P_{\hxi}(\hat v,z)|+ \frac{4}{\rho^2} \sup_{(\hat v,z) \in B_{\rho/2}} |f(\hat v,z)- P_{\he}(\hat v,z)|\\
&\le \frac{4}{\rho^2} \sup_{(\hat v,z) \in B_{\sqrt{3}\rho}(\hxi)} |f(\hat v,z)- P_{\hxi}(\hat v,z)|+ \frac{4}{\rho^2} \sup_{(\hat v,z) \in B_{\sqrt{3} \rho}(\he)} |f(\hat v,z)- P_{\he}(\hat v,z)|\\
&\le 8 (3)^{1+\alpha/2} C \rho^{\alpha}.
\end{align*}
Notice that 
\[
(P_{\hxi, \rho/2} - P_{\he, \rho/2}) (\hat v,z)=\dfrac{4}{\rho^2}[ a_\hxi - a_\he+ (b_\hxi- b_\he) \cdot \rho \hat v+ \rho^2 \hat v^T (C_\hxi- C_\he) \hat v+(d_\hxi-d_\he)  \cdot \rho z].
\]
Then by Lemma \ref{lm:abCdH} we get 
\begin{equation}
\label{eq:cdestH}
\begin{aligned}
|a_\hxi-a_\he| \le 2 (3)^{1+\alpha/2} M \rho^{2+\alpha} \quad &\text{and} \quad |b_\hxi-b_\he| \le 8 (3)^{1+\alpha/2} \sqrt{m}M \rho^{1+\alpha}\\
\|C_\hxi-C_\he \| \le 4 (3)^{1+\alpha/2} M \rho^{\alpha} \quad &\text{and} \quad |d_\hx-d_\hy| \le 4 (3)^{1+\alpha/2} M \rho^{\alpha}.
\end{aligned}
\end{equation}
By assumption \eqref{eq:CTEH} we easily get that $a_\hxi=f(\hxi)$, $f$ is continuous, $\hat X f(\hxi)=b_\hxi$, $Z f(\hxi)=d_\hxi$. Then by \eqref{eq:cdestH} we obtain that $\hat X f, Z f$ are continuous and $Z f$ is $C^{\alpha}$ . Moreover, since by Baker-Campbell-Hausdorff formula, see \cite[Theorem 15.1.1]{BLU}, we have
\[
\exp(s \hat X_i)(\exp(h \hat X_j)(\hxi))=\exp \left(s \hat X_i+h \hat X_j +\tfrac{sh}{2} [\hat X_j, \hat X_i] + O(s^2h)+ O(h s^2)\right) (\hxi),
\]
thanks to  \eqref{eq:CTEH} a straightforward computation shows  
\begin{align*}
&(f(\exp(s \hat X_i)(\exp(h \hat X_j)(\hxi)) )- f(\exp(h \hat X_j)(\hxi))-(f(\exp(s \hat X_i)(\hxi))-f(\hxi))\\
&=\left(2(C_\hxi)_{i,j} - \sum_{k=1}^n \frac{a_{i,j}^k}{2} Z_k f(\xi) \right)h s+O(s^{2+\alpha})+O(h^{2+\alpha}).
\end{align*}
Then there exists
\begin{align*}
&\lim_{h\to 0} \lim_{s\to 0} \frac{1}{h} \left( \frac{f(\exp(s \hat X_i)(\exp(h \hat X_j)(\hxi)) )- f(\exp(h \hat X_j)(\hxi))}{s}-\frac{f(\exp(s \hat X_i)(\hxi))-f(\hxi)}{s} \right)\\
&=2(C_\hxi)_{i,j} + \sum_{k=1}^n \frac{a_{i,j}^k}{2} Z_k f(\xi).
\end{align*}
On the other hand, letting $s\to 0$ in the previous limit we gain that
\[
\hat X_j \hat X_i f(\hxi)=\lim_{h \to 0 } \frac{\hat X_i f(\exp(h \hat X_j)(\hxi))-\hat X_i f(\hxi)}{h}=2(C_\hxi)_{i,j} - \sum_{k=1}^n \frac{a_{i,j}^k}{2} Z_k f(\xi).
\]
Therefore since $\sum_{k=1}^n a_{i,j}^k Z_k=[\hat X_i, \hat X_j]$ we obtain 
\[
\hat{X}^2 f(\hxi)_{i,j}=\frac{\hat X_i \hat X_j f(\hxi )+\hat X_j \hat X_i f(\hxi)}{2}=2(C_\hxi)_{i,j}
\]
Finally, by \eqref{eq:cdestH} we obtain that 
$|\hat{X}^2 f(\hxi)_{i,j}-\hat{X}^2 f(\he)_{i,j}|\le 8 (3)^{1+\alpha/2} C \tilde{d}(\hxi,\he)^{\alpha}$, for $i,j=2,\ldots,m$.
\end{proof}
\begin{remark}
\label{rk:midpointH}
Given two points $\hxi,\he \in \Pi$ such that $\rho=\tilde{d}(\hxi,\he)$ then there exists $\hat \zeta=(\frac{\hx+\hy}{2}, \frac{t+3\tau}{4}-\frac{1}{8}\escpr{\hat A \hx,\hy} )$ such that $\tilde{d}(\hat \zeta,\he)=\frac{\rho}{2} <\frac{\sqrt{3}}{2} \rho $ and $\tilde{d}(\hat \zeta,\hxi) <\frac{\sqrt{3}}{2} \rho$. Moreover if $\hxi,\he$ belongs to $B_R(\hxi_0)$ for $\hxi_0 \in \Pi$ then $\hat \zeta$ in $B_{2R}(\hxi_0)$.
\end{remark}

\begin{lemma}
\label{lm:abCdH}
Let $\hat v \in \Pi$ and  $P(\hat v,z)=a + b \cdot \hat v+ \hat v^T C \hat v+ d  \cdot z$, where $C$ is a symmetric matrix. Assume that there exists $M>0$ such that  $\|P\|_{L^\infty(B_1)} \le M$, then $|a|\le M$ and $\|C\|,|d| \le 2M$, $|b| \le 2\sqrt{m} M$.
\begin{proof}
Setting $\hat v=0$, $z=0$ we have $|a|\le M$. Let $\eps>0$, if $\hat v=0$, $z=\tfrac{d}{|d|(1+\eps)}$ we get $|a+ \frac{|d|}{1+\eps}| \le M$, thus $|d|\le 2M(1+\eps)$, letting $\eps \to 0 $ we get $|d|\le 2M$.  Let $\hat v^1,\ldots,\hat v^m$ be an orthonormal basis of $C$ with eigenvalue $\lambda_1,\ldots,\lambda_n$. Then for each $i$, setting $\hat v=\pm \tfrac{v^i}{(1+\eps)}$, $z=0$, we obtain 
\[
\left|\frac{b \cdot \hat v^i}{1+\eps}+ \frac{\lambda_i}{(1+\eps)^2}\right| \le 2M, \qquad \left|\frac{b \cdot \hat v^i}{1+\eps}-\frac{\lambda_i}{(1+\eps)^2}\right| \le 2M. 
\]
Then we have 
\begin{align*}
\frac{|b \cdot \hat v_i|}{1+\eps}\le \frac{1}{2} \left( \left|\frac{b \cdot \hat v^i}{1+\eps}+ \frac{\lambda_i}{(1+\eps)^2}\right|+ \left|\frac{b \cdot \hat v^i}{1+\eps}-\frac{\lambda_i}{(1+\eps)^2}\right| \right) \le 2 M\\
\frac{|\lambda_i|}{(1+\eps)^2}\le \frac{1}{2} \left( \left|\frac{b \cdot \hat v^i}{1+\eps}+ \frac{\lambda_i}{(1+\eps)^2}\right|+ \left|\frac{b \cdot \hat v^i}{1+\eps}-\frac{\lambda_i}{(1+\eps)^2}\right| \right) \le 2 M.
\end{align*}
Letting $\eps \to 0$ we get $||C||=\max_{i=1,\ldots,m}|\lambda_{i}| \le 2M$ and $|b|\le 2\sqrt{m}M$.
\end{proof}
\end{lemma}

\begin{lemma}
\label{lm:sypol2H}
Let $D_{0} \subset \Pi=\{v_1=0\}$ be a set axially  symmetric with respect  to $v_i=0$ for $i=2,\ldots,m$ and  $z_k=0$ for $k=1,\ldots,n$. Let 
\[
D_{\xi}=\exp(D_0)=\{\exp(\sum_{i=2}^m v_i \hat X_i + \sum_{k=1}^n z_k Z_k) (\xi) \ : \ (\hat v, z) \in D_0 \}
\]
and  $p$ is a polynomial of of degree less than or equal to $2$, 
\[
p ((\hat v,z))=a_0+ \sum_{i=2}^{m} a_i v_i +\sum_{k=1}^n b_k z_k+\sum_{i,j=2}^{m} c_{i,j} v_j v_i
\]
 Let $k$ be the kernel given by \eqref{eq:kkH}
then we have 
\begin{equation}
\label{eq:kp0H}
\int_{D_{\hat \xi}} k(\hat \xi, \hat \eta) p (\text{Log}_{\xi}( \eta)) \, d \hat \eta=0,
\end{equation}
where $\xi=(0, \hat \xi)$, $\eta=(0,\hat \eta)$ and $\text{Log}_{\xi}(\eta)=(0, -\hat \xi) \circ (0, \hat \eta)$
\end{lemma}
\begin{proof}
Changing the variable $(\hat v, z)=\Phi(\eta)=\text{Log}_{\xi}(\eta)$ we have  $\det(d\Phi)=1$. Therefore  we gain that the left hand side of \eqref{eq:kp0H} is equivalent to 
\begin{equation}
\label{eq:intproofoddlemma}
\int_{D_{0}} \hat k(\hat v, z) p (\hat v, z) \, d \hat v dz
\end{equation}
Since the kernel 
\[
\hat k((\hat v, z))= C_Q (2-Q) \dfrac{ 4 \sum_{i=2}^m \sum_{k=1}^n  a_{1,i}^k v_i z_k}
{  \Big( | \hat v |^4 + 
16 |z|^2 \Big)^{\frac{Q+2}{4}} }
\] 
is symmetric both in $v_i$ for $i=2\ldots,m$ and in $z_k$ for $k=1,\ldots,n$ the product between $\hat k$ and $p$ is still odd in some variables. Then the integral \eqref{eq:intproofoddlemma} vanishes on the axial symmetric domain $D_0$. 
\end{proof}

\begin{theorem}
\label{th:KcontH}
Let $k$ be the kernel defined in \eqref{eq:kkH}, we set  
\[
K(f)( \hat \xi)=\int_{\Pi} k(\hat \xi, \hat  \eta) f(\hat \eta) d \hat \eta
\]
for each $\hat \xi \in \Pi$. Assume that $f \in C^{2,\alpha}(\Pi)$ and $f$  compactly supported in $\Pi$ then there exists a constant $C$ such that 
\begin{equation}
\label{eq:C2aestimateH}
\| K (f) \|_{C^{2,\alpha}(\Pi)} \le C \| f \|_{C^{2,\alpha}(\Pi)}
\end{equation}
\end{theorem}

\begin{proof}
First of all we notice that $\Pi$ with the law induced by $\GG$ is an homogenous group of homogeneous dimension $Q-1$. Setting $(0,\hat v, z)=\text{Log}_{(0,\hxi)}(0,\he)$ we have that $\hat k (\hat v, z)$ defined in \eqref{eq:kkHO} is $C^{\infty}(\Pi \smallsetminus {\{0\}})$, homogeneous of degree $1-Q$ and thanks to Lemma \ref{lm:sypol2H} defines a singular integral on $\Pi$.  Following \cite[Section 3]{Jerison1} or \cite[p. 32]{NagelStein79} there exists a linear map $L_{\hxi}$ such that $L_{\hxi}(\hxi- \he)=(0,\he)^{-1} \circ (0,\hxi)$. Denote $\tilde{L}_{\hxi}= (L_{\hxi}^{-1})^T$ then  $K$ is realized as a pseudo-differential operator with symbol $a(\hxi, \zeta)=\mathcal{F}(\hat k)(\tilde{L}_{\hxi}(\zeta))$ where $\mathcal{F}$ denotes the Fourier transform and $\hat k$ does not  refer to the Fourier  transform, but to definition \eqref{eq:kkHO}. Since $\hat k$ is of class $1-Q$, by \cite[Theorem 1, p. 9]{NagelStein79} we have $\mathcal{F}(\hat k)$ is of class $0$. Therefore the symbol $a$ belongs to the class $\mathcal{S}_{\tilde d}^0$, see \cite[p. 56]{NagelStein79}. Then by \cite[Theorem 13, p. 83]{NagelStein79} we get that $a(\hxi, D)$ (that coincides with $K$) is bounded from $\Gamma^{2,\alpha}(\Pi)$ to $\Gamma^{2,\alpha}(\Pi)$. Hence, by Proposition \ref{prop:Gamma=C2H} we obtain the desired estimates \eqref{eq:C2aestimateH} in $C^{2,\alpha}(\Pi)$.
\end{proof}

\subsection{The reflection technique}
\begin{definition}
Let $g$ be a function in $\Pi$ and $x \in \hh^1 \smallsetminus \Pi$. We set 
$$
\tilde K_1(g) (\xi) = \int_{\Pi} \tilde k_1(\xi,\he) g(\he) d\sigma(\he), \qquad \tilde K(g) (\xi) = \int_{\Pi} \tilde k(\xi,\he) g(\he) d\sigma(\he)
$$
where
\begin{equation}
\label{eq:tidek1H}
\tilde k_1(\xi,\he)=
C_Q (Q-2) \dfrac{ |x-(0,\hat y)|^2 x_1  }
{  \Big( \big| x-(0,\hat y) \big|^4 + 
16 \sum_{k=1}^n \big(\tau_k - t_k - \frac{1}{2}  \escpr{ \hat A^{(k)}\hx,\hy)} \big)^2 \Big)^{\frac{Q+2}{4}}}
\end{equation}
and
\begin{equation}
\label{eq:tidekH} \tilde k(\xi,\he)=
 C_Q (2-Q) \dfrac{ 4 \sum_{i=2}^m \sum_{k=1}^n  \big(\tau_k - t_k - \frac{1}{2}  \escpr{\hat A^{(k)} \hx, \hat y } \big) a_{1,i}^k (y_i-x_i)}
{  \Big( \big| x-(0,\hat y) \big|^4 + 
16 \sum_{k=1}^n \big(\tau_k - t_k - \frac{1}{2}  \escpr{ \hat A^{(k)} \hx, \hat y} \big)^2 \Big)^{\frac{Q+2}{4}}},
\end{equation}
\end{definition}

\begin{remark}
\label{rk:tiledekH}
Notice that $\tilde k(\xi,\he)$ defined in \eqref{eq:tidekH} converges to the convolution kernel $k(\hxi,\he)$ defined in \eqref{eq:kkH} as one approaches the boundary.
\end{remark}

\begin{lemma}\label{tildelimitH}
Let  $g$ be  a Lipschitz compactly supported function in 
$\Pi$ and $\hxi_0$ be a point in $\Pi$.  
For $x \in \hh^1 \smallsetminus \Pi$ we consider
$$
\tilde{K}_1(g) (\xi) = \int_{\Pi} \tilde{k}_1(\xi,\he) g(\he) d\sigma(\he),
$$
where $\tilde{k}_1$ is defined in \eqref{eq:tidek1H}.
Then we  have
\begin{align*}
 &\tilde{K}_1(g) (\xi) \to \frac{1}{2} g(\hxi_0) \quad \text{ as } \xi \to \hxi_0^+,\\
 &\tilde{K}_1(g) (\xi) \to -\frac{1}{2} g(\hxi_0) \quad  \text{ as } \xi \to \hxi_0^-,
\end{align*}
so that $(K_1)^+=\tfrac{1}{2}\text{Id}$ while restricted to $\Pi$ and $(K_1)^-= - \frac{1}{2} \text{Id}$ while restricted to $\Pi$.
 \end{lemma}
\begin{proof}  
By Proposition \ref{frompositiveH} $K_1 (g) (\xi)$ converges to $\pm \tfrac{1}{2} g$ and since 
\begin{align*}
&|K_1 (g) (\xi)- \tilde{K}_1(g)(\xi)| \\
&\leq  \sup_{\hxi, \he} \left|1- \tfrac{\Big( \big| x-(0,\hat y) \big|^4 + 
16 \sum_{k=1}^n \big(\tau_k - t_k - \frac{1}{2}  \escpr{ A^{(k)}x,(0,\hat y)} \big)^2 \Big)^{\frac{Q+2}{4}}}{\Big( \big| x-(0,\hat y) \big|^4 + 
16 \sum_{k=1}^n \big(\tau_k - t_k - \frac{1}{2}  \escpr{ \hat A^{(k)}\hx,\hy)} \big)^2 \Big)^{\frac{Q+2}{4}}} \right| K_1(g)(\xi)
\end{align*}
we have $K_1 (g) (\xi)- \tilde{K}_1(g)(\xi)$ goes to zero when $x_1$ tends to $0$. Then also
$\tilde{K}_1(g)(\xi)$ converges to $\tfrac{1}{2} g$ when $x_1\to0^+$ and $\tilde{K}_1(g)(\xi)$ converges to $-\tfrac{1}{2} g$ when $x_1\to0^-$ .
\end{proof}

Given $r \in \rr$, let us denote $\Pi_r=\{\xi=(r, \hx, t)\}$. We consider the $C^{2, \alpha}(\Pi_r)$ norm with respect to the distance $\tilde{d}$  as we did in Section  \ref{sc:c2alphaestimeteH}. This choice allows us to completely decouple variables and we have

\begin{proposition}\label{IKnormH}  
Let $\Pi=\{ x_1=0\}$ and $K$ the singular operator defined by the kernel $k$, see \eqref{kH}. Then we have
$$||(-\frac{1}{2}I + K)(g)||_{C^{2, \alpha}(\Pi)} = ||(\frac{1}{2}I + K)(g)||_{C^{2, \alpha}(\Pi)}.$$
\end{proposition}
\begin{proof}
Since the $C^{2,\alpha}$ norm on $\Pi_r$  with respect to the distance $\tilde{d}$ are independent on $r$, we have
$$||(\tilde K_1 + \tilde K)(g)(- \cdot, \cdot, \cdot)||_{C^{2, \alpha}(\Pi_{-r})} = ||(\tilde K_1 + \tilde K)(g)||_{C^{2, \alpha}(\Pi_r)}.$$
Letting $r$ to 0, and applying Lemma \ref{tildelimitH} and Remark \ref{rk:tiledekH} we get the thesis.
\end{proof}

All the results in Section \ref{sc:methodcontinuity} follow in the same way and we obtain that $\tfrac{1}{2} I +  K $ is invertible from $C^{2,\alpha}(\Pi)$ to $C^{2,\alpha}(\Pi)$.

\subsection{The Poisson kernel and Schauder estimates}
Once we have the invertibility of $\tfrac{1}{2} I +  K $, we consider the Poisson kernel \eqref{eq:Pg} and 
 the analogous of Theorem \ref{c:schauderGroups} for a bounded domain $\Omega \subset \GG$  follows in the same way, clearly replacing $\Delta_{\hh^1}$ with $\Delta_{\GG}$ . The definition of smooth domain is the same of Definition \ref{def:smoothboundary} and we say that $\xi$ in $\partial \Omega$ is a \textit{characteristic point} if $\nabla_{\GG} \psi (\xi)=0$, where $\psi$ is  the defining function of the boundary defined in  \ref{def:smoothboundary}.

\begin{theorem} \label{c:schauder2H} 
Let $\Omega \subset \GG$ be a smooth bounded domain and $u$ is the unique
solution to
$$\Delta_{\GG} u=f\; \text{in}\  \Omega, \quad u= g \text{ on }\,  \partial \Omega, $$
where $f \in C^\alpha(\bar \Omega)$ and $g \in  \Gamma^{2, \alpha} (\partial \Omega)$ and $0<\alpha<1$. 
Let $\bar \xi \in  \partial\Omega $ be a non-charateristic point, $V\subset \GG$ be an open neighborhood of $\bar \xi $ without charateristic points and $\phi\in C^\infty_0(V)$ be a bump function equal to $1$ in neighborhood $V_0 \subset \subset V$ of $\bar \xi$. Then we have 
$\phi u \in C^{2, \alpha}(\bar \Omega\cap V)$ and 
\begin{equation}
	\label{stimeH}
\|\phi u\|_{C^{2, \alpha}(\bar \Omega\cap V)} \leq C (\|g\|_{ \Gamma^{2, \alpha} (\partial \Omega)} + \|f\|_{C^\alpha(\bar \Omega)}).\end{equation}
\begin{proof}
Let us denote by $\Omega$ a smooth, open bounded set in $\GG$ and let 
$0\in \partial \Omega$ be a non characteristic point.  
The boundary of $\Omega$ can be identified in a neighborhood $V$ with the graph of a regular 
function $w$, defined on a neighborhood $\hat V=V\cap \rr^{m+n-1}$ of $0$: 
$$\partial \Omega \cap V= \{(w(\hat s), \hat s): \hat s\in \hat V  \}.$$
We can as well assume that $w(0) = 0$,  $\nabla w=0$. This implies that 
\begin{equation}\label{tuttoqui}
w(\hat s) = O(|\hat s|^2) 
\end{equation}
as $\hat s \to 0$. 
On the set $V$ the function 
$\Xi(s_1, \hat s) = (s_1 - w(\hat s), \hat s) $
is a diffeomorphism. It sends  $\partial \Omega\cap V$ to a subset of the plane $\{x_1 =0\}$:
$$\Xi(\partial \Omega \cap V) =\{(x_1, \hxi): x_1 =0\}= \Pi_{\Xi}.$$
Moreover, we have 
\begin{equation}\label{ohscusa}\Delta_\Xi =  d\Xi(\Delta_{\hh^1}), \end{equation}
with fundamental solution 
$$ \Gamma_\Xi(\xi) = \Gamma(x_1 + w(\hxi), \hxi).$$
For $x_1$ small enough we have
$$ \Gamma_\Xi(x_1, \hxi) = \Gamma(x_1+w(\hxi), \hxi)
= \Gamma(x_1, \hxi)
 +  R(x_1, \hxi), $$
 where 
 $$  R(x_1,\hxi)=w(\hxi) \nabla \Gamma( x_1 +\sigma w(\hxi), \hxi) $$
 for some $\sigma \in (0,1)$. Furthermore we have that 
 $$X_{1,\Xi}^{\xi}=d\Xi(X_1^{\xi})=X_1^{\xi}-\frac{1}{2}\sum_{k=1}^n a_{1,i}^k x_i \partial_{t_k} w(\hxi) \partial_{x_1}.$$
 Notice that $\Gamma$ is a rational function that goes as $d^{-Q+2}$, its first derivatives go as $d^{-Q+1}$ and its second derivatives go as $d^{-Q}$. On the other side the function $w(\hxi)$ has a 0 of order 2 thus $w(\hxi)$ goes as $d^{2}$. Then we have 
 \[
 X_{1,\Xi}^{\eta} \Gamma_\Xi(0, \he)=X_1^{\eta} \Gamma (0,\he)+ \hat R(0,\he),
 \]
 where 
 \[
 \hat R(0,\he)=X_1^{\eta} R(0,\he)-\frac{1}{2}\big( \sum_{k=1}^n a_{1,i}^k y_i \partial_{t_k} w(\he) \big) \partial_{y_1} R(0,\he) -\frac{1}{2}\big( \sum_{k=1}^n a_{1,i}^k y_i \partial_{t_k} w(\he) \big)\partial_{y_1} \Gamma(0,\hat y)
 \]
 that goes as 
 \[
 |\hat R(0,\hxi)| \le \hat d^{-Q+2}(0,\hxi),
 \]
 where $\hat d$ is the induce distance.
Therefore the operator  $K_{\hat R}$ with kernel $\hat R$ is compact since the homogeneous dimension of the boundary is $Q-1$. Finally, the proof ends  following the same steps of the proof of Theorem \ref{c:schauder2} and using Proposition \ref{prop:K1KplaneHT} instead of Proposition \ref{prop:K1Kplane}.
\end{proof}
\end{theorem}

Finally, the analogous of Corollary \ref{c:schauder3} for a bounded domain $\Omega \subset \GG$ without characteristic points  follows in the same way, clearly replacing $\Delta_{\hh^1}$ with $\Delta_{\GG}$ .

\bibliographystyle{abbrv} 
\bibliography{bib}

\end{document}